\documentclass[journal,twoside,web]{ieeecolor}
\usepackage{generic}
\usepackage{cite}
\usepackage{amsmath,amssymb,amsfonts}
\usepackage{graphicx}
\usepackage{textcomp}
\usepackage{comment}
\usepackage{algorithm}
\usepackage{algpseudocode}
\usepackage{soul}
\usepackage{setspace}

\newcommand{\fc}[1]{\textcolor{blue}{#1}}

\usepackage{preamble_papers}

\def\BibTeX{{\rm B\kern-.05em{\sc i\kern-.025em b}\kern-.08em
    T\kern-.1667em\lower.7ex\hbox{E}\kern-.125emX}}
\begin{document}
\title{Learning-Based Stochastic Optimal Control with Infinite-Horizon Probabilistic Constraints}
\author{Francesco Cordiano, Kanghui He, and Bart De Schutter, \IEEEmembership{Fellow, IEEE}
    \thanks{This project has received funding from the European Research Council (ERC) under the European Union's Horizon 2020 research and innovation programme (Grant agreement No.\ 101018826 - ERC Advanced Grant CLariNet), and by the Rubicon Postdoctoral Fellowship (Correspondence No. 2026/ENW/02250137), funded by the Netherlands Organisation for Scientific Research (NWO).}
    \thanks{Francesco Cordiano and Bart De Schutter are affiliated with the Delft Center for Systems and Control, Delft University of Technology, 2628 CD, Delft, The Netherlands, email: $\{$f.cordiano, b.deschutter$\}$@tudelft.nl. Kanghui He is affiliated with the Department of Engineering Science, University of Oxford, OX1 3PJ Oxford, U.K.,        
    email: kanghui.he@eng.ox.ac.uk.}
}


\maketitle

\begin{abstract}
In this paper, we consider stochastic optimal control problems with infinite-horizon joint chance constraints. By means of an appropriate state augmentation, we reformulate the original problem as a constrained Markov decision process, in which both the cost and the constraint function exhibit an additive structure. We then prove that this formulation enjoys strong duality, thereby enabling us to reformulate the problem as an equivalent unconstrained one in the Lagrange dual framework. We propose a dual-ascent algorithm to solve the resulting problem and show that it converges to a deterministic Markov policy defined over the augmented state space that is both optimal and feasible. To accommodate continuous state-input spaces, we propose a dedicated learning algorithm to approximate the value function in an offline training setting, thereby significantly reducing the computational complexity of the online control phase. We then test our approach on a numerical example and demonstrate its effectiveness compared to online predictive control methods in terms of performance and computational complexity.

\end{abstract}


\section{Introduction}
\label{sec:introduction}
Decision-making problems are central to many domains, including energy, transportation, robotics, and finance \cite{mesbah2016stochastic, pippia2021scenariobased, dariano2019integrated, bemporad2010scenariobased}. In these settings, the control action must account for potential sources of uncertainty, such as external disturbances or modeling errors, to ensure that both performance and safety requirements are satisfied.

Stochastic optimal control \cite{bertsekas1996stochastic} provides a flexible mathematical framework to synthesize optimal policies for multi-stage decision-making problems.
Such problems can be solved either through offline methods, such as dynamic programming (DP) \cite{bertsekas2012dynamic}, or via online finite-horizon approximations, as commonly done in model predictive control (MPC) \cite{paulson2019efficient}. Although both approaches allow constraints to be considered in the problem formulation, their applicability can be significantly limited for complex constraint structures.
For example, in the context of stochastic safety-critical systems, it is common
to require that the controlled system satisfies constraints along
the entire path with a certain (high) probability \cite{ono2008iterative, wang2022solving, schmid2025computing, laurenti2025unifying}, typically referred to as joint chance constraints, or mission-wide chance constraints. This is in contrast to individual (or stagewise) chance constraints, where, given the current state measurement, we require that the constraint is satisfied with a certain probability for the next state only \cite{farina2016stochastic}. Since joint chance constraints involve the entire path, the resulting policy can in general be non-Markovian and may even be stochastic, as observed in \cite{schmid2025computing, ni2025learning}.

Historically, joint chance constraints have been addressed using Boole's inequality \cite{farina2016stochastic}, which allows replacing a joint chance constraint over $N$ steps with $N$ individual chance constraints, by appropriately adjusting the risk parameter. This can introduce conservatism, which can be mitigated by iterative risk allocation approaches \cite{ono2008iterative}. Alternatively, joint chance constraints can be encoded in MPC problems by treating the system trajectory over $N$ steps jointly as a random variable and approximating the resulting chance constraint using suitable methods, e.g., randomized or scenario-based approaches \cite{prandini2012randomized, cordiano2024scenario}. However, although this guarantees open-loop satisfaction of chance constraints, assessing the actual probability of constraint violation for the closed-loop trajectory is difficult due to the inherent receding-horizon nature of MPC. 

Recent papers \cite{wang2022solving, schmid2025computing, ni2025learning} 
have dealt with the inherent non-Markovian nature of joint chance constraints by introducing an appropriate state augmentation. This allows avoiding the usage of 
Boole's inequality, but existing works allow only finite-horizon formulations.
Recent advances in constrained reinforcement learning \cite{paternain2019constrained, paternain2023safe, chen2024probabilistic} propose policy-gradient schemes leveraging known results about strong duality for constrained MDPs \cite{hernandez-lerma2000constrained}, but similarly assume that constraint violations do not occur after a known time step, or employ Boole's inequality for computational tractability.

On the other hand, joint chance constraints in the infinite-horizon setting have received much less attention. For example, a related problem is considered in \cite{abate2007probabilistic}. There, the goal is to maximize the probability that the system remains in a safe set over an infinite horizon, but performance is not encoded in the resulting DP scheme, and the resulting policy is the one that aims at maximizing safety. In \cite{riccardi2026temporal} safety is encoded via abstraction of the safety specification and the system dynamics, and an MPC problem is solved online for performance optimization. However, the resulting approach can become computationally expensive for high-dimensional systems or complex nonlinear dynamics.
Alternatively, stochastic or probabilistic barrier functions allow one to optimize performance while considering probabilistic safety constraints, but they may yield suboptimal and conservative policies, and the probability of safety typically decreases linearly with the prediction horizon \cite{laurenti2025unifying, airaldi2025probabilistically}.

In view of this, the computation of a policy for infinite-horizon joint chance-constrained problems without resorting to conservative approximations is still an open challenge. In this paper, we consider an infinite-horizon stochastic optimal control problem with a discounted cost function and joint chance constraints. We extend \cite{wang2022solving, schmid2025computing, ni2025learning} to an infinite-horizon setting, and in contrast to \cite{ono2008iterative, paternain2023safe, chen2024probabilistic}, we do not resort to conservative approximations based on Boole's inequality, nor do we assume that violations are possible only before a known time step. To allow for general continuous state-action spaces, we propose an algorithm to learn the value function of the given problem using suitable function approximators in an offline step. The novel contributions of our paper are:

\subsubsection{Equivalent Markovian formulation} 
Similar to the finite-horizon case \cite{wang2022solving, schmid2025computing, ni2025learning}, we design a novel state augmentation to cast the infinite-horizon chance-constrained problem as a constrained MDP \cite{hernandez-lerma2000constrained, altman2021constrained}.
In view of this formulation, we discuss necessary and sufficient conditions for the feasibility of the original problem and observe that the usual assumption of bounded uncertainty support can, in principle, be removed in the infinite-horizon setting as well.

\subsubsection{Value computation via duality}
We show that the resulting formulation enjoys strong duality, enabling us to reformulate the problem as an equivalent unconstrained one using the Lagrangian dual framework, where the chance constraint is treated as a penalty term in the objective.
We propose an offline dual-ascent algorithm to compute the value function of the chance-constrained problem. This consists of a primal step, in which we solve a model-based unconstrained problem via dynamic programming, followed by the update of the dual variable. In contrast to previous works, which establish feasibility either via mixed policies \cite{schmid2025computing} or in average over the learning iterations \cite{chen2024probabilistic}, we show that our algorithm converges to a deterministic Markov policy that is both optimal and feasible.

\subsubsection{Learning-based framework}
To tackle general continuous state-action spaces, we propose a dedicated learning scheme based on approximate dynamic programming. Specifically, we extend recent learning-based frameworks \cite{he2024approximate, he2024approximatea} to approximate the Lagrange dual function of the chance-constrained problem in an offline learning phase, which is then used to estimate the value function of the original problem. To ensure high learning quality, we first show that, under mild assumptions, the dual function of the problem of interest enjoys a continuity property, thereby justifying the use of, e.g., neural network approximations, which have universal approximation properties for continuous functions \cite{hornik1989multilayer}.

The paper is organized as follows: Section \ref{sec:problem_formulation} introduces the problem of interest and preliminary mathematical concepts. In Section \ref{sec:tractable}, we reformulate the infinite-horizon chance-constrained problem using Lagrange duality in an appropriate augmented state space. In Section \ref{sec:dual_ascent} we propose a dual-ascent algorithm to solve the resulting problem, and we present related convergence and feasibility guarantees. In Section \ref{sec:learning} we propose a dedicated learning algorithm to approximate the value function of the given problem. Finally, in Section \ref{sec:experiment} we demonstrate the effectiveness of our method on a numerical case study, and Section \ref{sec:conclusions} concludes the article.

\section{Preliminaries and Problem formulation}
\label{sec:problem_formulation}

\subsection{Preliminaries and notation}\label{subsec:preliminaries}
In this paper, $\Z_{\geq a}$ denotes the set of integers greater than or equal to a given constant $a\in\Z$, and $ \R_{\geq a}$ 
denotes the set of real numbers greater than or equal to a given constant $a\in\R$.
The symbol $\mathbf{1}_{X}$ denotes the indicator function of the logical condition $X$, taking the value 1 if $X$ is true, and 0 otherwise. For any $\lambda\in\R$, we denote its projection over $\R_{\geq0}$ as 
$$[\lambda]_{\geq0} = \begin{cases}
    \lambda \ \text{if } \lambda \geq 0
    \\ 0  \ \text{otherwise}.
\end{cases}$$

Let us consider a dynamical system, with dynamics
\begin{align}\label{eq:sys}
    x_{t+1} = f(x_t, u_t, d_t)
\end{align}
where, $\forall t\in\Z_{\geq0}$,  $x_t\in\R^n$ is the state of the system, $u_t\in\R^m$ is the control input, $d_t\in\Delta$ represents exogenous uncertainty described by a certain time-invariant probability distribution with domain $\Delta\subseteq\R^q$, and $f:\R^n\times\R^m\times\R^q\to\R^n$. 
For $t\in\Z_{\geq0}$, we denote the space of histories up to time step $t$ as $\Hcal_t = (\R^n\times\R^m)^{t-1}\times\R^n$, such that $\Hcal_t\ni h_t=(x_0,u_0,...,u_{t-1}, x_t)$, with $x_k\in\R^n, \forall k\in\{0,...,t\}$ and $u_k\in\R^m, \forall k\in\{0,...,t-1\}$.

Let $\Ucal\subseteq\R^m$ be the space of admissible inputs.
We define $\Gamma$ as the space of stochastic policies $\pi$ on $\Ucal$, where $\pi:=\{\mu_t\}_{t=0}^\infty$, and $\mu_t:\Ucal\times \Hcal_t\to[0,1]$ is a Borel-measurable stochastic kernel that assigns a probability measure $\mu_t(\cdot|h_t)$, for a given history $h_t\in\Hcal_t$. Similarly, a deterministic policy is such that $\mu_t(\cdot| h_t) = \delta_{u_t}(\cdot)$, where for a given $u\in\Ucal$, $\delta_u:\Ucal\to\{0,1\}$ is the delta-Dirac function that outputs 1 at $u$ and 0 otherwise. A policy is Markovian if $\mu_t(\cdot|h_t)=\mu_t(\cdot|x_t)$. For simplicity, we refer to $\mu_t$ as the deterministic map from $h_t$ to $u_t$ when $\mu_t$ is a deterministic policy.
Note that, although the previous definitions are introduced for a dynamical system with a continuous state-input space, they can be extended to more general cases, e.g., by including discrete states or actions as well.

Last, for the system \eqref{eq:sys}, we denote by $\P_d^\pi$ and $\E_d^\pi$, respectively, the probability measure and expectation operator induced by the exogenous uncertainty $d_t, t\in\Z_{\geq0}$ and by the stochastic
policy $\pi$.

\subsection{Problem statement}\label{subsec:statement}
We consider the following infinite-horizon chance-constrained optimal control problem, with a given initial state $x_0\in\R^n$:
\begin{align}\label{cc_ocp}
    \begin{split}
        \inf_{\pi\in\Gamma} \ & \E_d^\pi\left[ \sum_{t=0}^\infty \gamma^t \ell(x_t, u_t) \right]
        \\ \text{s.t.} \ &  x_{t+1} =  f(x_t, u_t, d_t), \ \forall t\in\Z_{\geq0}
        \\& \P_d^\pi(x_t\in\Xcal, \ \forall t\in\Z_{\geq0})\geq1-\veps
        \\& u_t\sim\mu_t(\cdot|x_0,u_0,...,x_t), \ \forall t\in\Z_{\geq0}
        \\& \pi=(\mu_0, \mu_1, ...),
    \end{split}
\end{align}
where $\gamma\in(0,1)$
is a discount factor, and $\veps\in[0,1)$ is a risk parameter. In \eqref{cc_ocp}, we require that the infinite-horizon system performance is optimized, while ensuring that the system state belongs to a constraint set $\Xcal\subseteq\R^n$ with probability at least $1-\veps$, jointly over the entire trajectory.
Throughout the paper, we consider the following assumption:
\begin{assumption}\label{ass:cont}
    The functions $f:\R^n\times\R^m\times\R^q\to\R^n$ and $\ell:\R^n\times\R^m\to\R_{\geq0}$ are continuous, and the sets $\Xcal$ and $\Ucal$ are compact. In addition, let $\Xcal_\text{feas}\subseteq\Xcal$ be the set in which \eqref{cc_ocp} is feasible. We assume $\Xcal_\text{feas}$ to be non-empty.
\end{assumption}

The continuity of the dynamics $f$ and of the stage cost $\ell$, as well as the compactness of $\Xcal$ and $\Ucal$, are standard in stochastic optimal control problems and in the MDP literature \cite{altman2021constrained, hernandez-lerma2000constrained}, and it is typically needed to ensure the existence of a minimizing policy. 

Problems like \eqref{cc_ocp} are notoriously difficult: indeed, stochastic constraints generally require non-deterministic policies, and the joint-in-time structure of the chance constraints may induce a non-Markovian structure of the resulting optimal policy. Therefore, classical value-based or policy-based algorithms cannot be directly applied to \eqref{cc_ocp}, which first requires suitable reformulations.


\section{Tractable reformulation}\label{sec:tractable}
As stated before, the challenge in \eqref{cc_ocp} is twofold. On one hand, we need to account for the non-Markovian structure of the problem; on the other hand, the chance constraint in \eqref{cc_ocp} prevents us from applying standard value-based iterative algorithms.
To address the two challenges, we now propose a Markovian problem equivalent to \eqref{cc_ocp} through appropriate state augmentation, and cast it in an equivalent unconstrained formulation in terms of its Lagrange dual.

\subsection{State augmentation}\label{sec:state_augmentation}
As we observed, a feasible policy for \eqref{cc_ocp} is, in general, history-dependent. This is because the joint-in-time chance constraint involves the entire system trajectory rather than the current state realization $x_t$. Furthermore, the presence of a discount factor in the cost function can create a temporal mismatch, since the cost function is discounted but no discount appears in the constraint function, potentially making the optimal policy in \eqref{cc_ocp} time-varying.
In the following, we solve these problems by introducing three auxiliary states. More specifically, we define the augmented state space $\Scal:=\R^n\times\{0,1\}^2\times[0,1]$, with $s:=[x^\top, \xi, \psi, \phi]^\top$. The dynamics of the state $s$ consist of the following updates, defined by $F:\Scal\times\R^m\times\R^q\to\Scal$:
\begin{align}\label{eq:aug_dynamics}
s_{t+1} = F(s_t, u_t, d_t) :=
    \begin{cases}
        x_{t+1} = f(x_t, u_t, d_t)
        \\ \xi_{t+1} = \xi_t \mathbf{1}_{x_{t+1}\in\Xcal}
        \\ \psi_{t+1} = \xi_t - \xi_{t+1}
        \\ \phi_{t+1} = \gamma\phi_t   ,
    \end{cases}
\end{align}
with
$$s_0 = [x_0^\top, \ \mathbf{1}_{x_0\in\Xcal}, \ \mathbf{1}_{x_0\not\in\Xcal}, \ \phi_0]^\top,$$
and $x_0\in\R^n, \phi_0=1$.
As we show now, the first two additional states $\xi$ and $\psi$ are introduced to make the problem Markovian. Conversely, the purpose of $\phi$ is to store the time-dependence due to the discount factor, eliminating the temporal mismatch between cost and constraint function. In fact, by initializing $\phi_0=1$, we have $\phi_t=\gamma^t$.

The state $\xi$ was already defined in \cite{wang2022solving} and \cite{schmid2025computing}, where it is observed that $\xi_t$ takes the value 1 if and only if the entire path $(x_0, ..., x_t)$ satisfies the constraint. Moreover, we have
$$\E_d^\pi[\xi_t] = \E_d^\pi \left[\Pi_{k=0}^t \mathbf{1}_{x_k\in\Xcal}\right] = \P_d^\pi(x_k\in\Xcal, \forall k\in\{0,...,t\}).$$

However, \cite{wang2022solving} and \cite{schmid2025computing} consider the finite-horizon setting, in which case the chance-constraint is simply substituted by $\E_d^\pi[\xi_N]\geq 1-\veps$, where $N\in\Z_{\geq0}$ is the (finite) prediction horizon. However, in our case, we consider the infinite-horizon setting. For this reason, in this paper, we introduce the additional state $\psi_t, t\in\Z_{\geq0}$, which, in view of \eqref{eq:aug_dynamics}, takes the value 1 if and only if $x_t$ is the first state that violates the constraint, and 0 otherwise.  In fact, for all $t\in\Z_{\geq1}$, we have:
\begin{align}
    \E_d^\pi[\psi_t] 
    & = \E_d^\pi\left[\xi_{t-1} - \xi_t \right] \nonumber
    \\& = \E_d^\pi\left[\mathbf{1}_{x_{t}\not\in\Xcal} \xi_{t-1}\right]  \nonumber
    \\& = \E_d^\pi\left[\mathbf{1}_{(x_{t}\not\in\Xcal\land x_{t-1}\in\Xcal\land ... \land x_0\in\Xcal)}\right]  \nonumber
    \\& = \P_d^\pi(x_t\not\in\Xcal \land x_k\in\Xcal,  \forall k\in\{0,...,t-1\}), \label{eq:psi_interpret}
\end{align}
which follows directly from the definitions of the states $\xi$ and $\psi$ in \eqref{eq:aug_dynamics} and from the fact that the expectation of the indicator function of an event is the probability of that event. Note that \eqref{eq:psi_interpret} means that $\E_d^\pi[\psi_t]$ is the probability that $x_t$ is the first state that violates the constraint, for a certain $t\in\Z_{\geq0}$. Since the events in the probability operator in \eqref{eq:psi_interpret} are disjoint for $t\in\Z_{\geq1}$, by initializing $\xi_0=1$ and $\psi_0=0$ we have\footnote{By noticing that $\sum_{t=0}^\infty \E_d^\pi[|\psi_t|]$ is bounded in view of \eqref{eq:aug_dynamics}, expectation and summation can be swapped in view of Fubini's Theorem (see, e.g., \cite{rudin1974real}, Chapter 8).} 
\begin{align}
        \sum_{t=0}^\infty \E_d^\pi[\psi_t] 
        & = \sum_{t=0}^\infty \P_d^\pi(x_t\not\in\Xcal \land x_k\in\Xcal, \forall k\in\{0,...,t-1\}) \nonumber
        \\& = \P_d^\pi(\exists t\in\Z_{\geq0}: x_t\not\in\Xcal) \nonumber
        \\& = 1-\P_d^\pi(x_t\in\Xcal, \forall t\in\Z_{\geq0}). \label{eq:psi_constraint}
\end{align}
Therefore, the constraint
\begin{align}\label{eq:cc_psi}
    \E\left[\sum_{t=0}^\infty \psi_t \right] \leq \veps
\end{align}
is equivalent to the chance constraint in \eqref{cc_ocp}.

We highlight two important aspects of this formulation, which rely on the additional binary states: First, the purpose of the state $\xi_t$ is essentially to memorize the history of the system until time step $t\in\Z_{\geq 1}$; second, the state $\psi_t$ allows to formulate the chance constraint as the additive constraint \eqref{eq:cc_psi}, which, as we will see in the next sections, yields computational advantages.

\subsection{Feasibility of \eqref{cc_ocp}}
The reformulation \eqref{eq:psi_interpret}--\eqref{eq:cc_psi} of the chance constraint unveils necessary and sufficient conditions to ensure the feasibility of \eqref{cc_ocp}. At first sight, one might argue that \eqref{cc_ocp} can be feasible only if the disturbances have a bounded support \cite{laurenti2025unifying}. However, note that this is not strictly necessary. In fact, in view of \eqref{eq:psi_constraint}, we see that the infinite-horizon probability of constraint violation is smaller than $\varepsilon$, with $\veps\in(0,1)$ if and only if 
$\E\left[\sum_{t=0}^\infty \psi_t \right]\leq \veps$,
which is the case only if there exists a policy $\pi\in\Gamma$ such that
\begin{align}\label{eq:feasibility_condition}
    \P_d^\pi(x_t\not\in\Xcal \land x_k\in\Xcal, \forall k\in\{0,...,t-1\}) \underset{t\to\infty}{\longrightarrow}0
\end{align}
fast enough, i.e., the probability that $x_t$ is the first state that violates the constraint decays to 0 over time. This can be the case even when the disturbances have an unbounded support. For example, consider a dynamical system that is stabilizable in the mean-square sense \cite{bernardini2012stabilizing}, i.e., $\exists c\geq0, \rho \in[0,1): \E_d^\pi[\|x_t\|_2^2]\leq c\rho^t\|x_0\|_2^2$.
Then, assuming that the origin is in the interior of the feasible set, we have $x_t\not\in\Xcal\Rightarrow\|x_t\|_2^2\geq a$, for a certain $a>0$. Therefore:
\begin{align*}
    &\P_d^\pi(x_t\not\in\Xcal \land x_k\in\Xcal, \forall k\in\{0,...,t-1\})
    \\&\leq \P_d^\pi(x_t\not\in\Xcal)
    \\& \leq 
    \P_d^\pi(\|x_t\|_2^2 \geq a) 
    \\& \leq \frac{\E_d^\pi[\|x_t\|_2^2]}{a}  
    \\& \leq \frac{c\rho^t\|x_0\|_2^2}{a} \underset{t\to\infty}{\longrightarrow} 0,
\end{align*}
where we have used the Markov inequality \cite{durrett2019probability} applied to the random variable $\|x_t\|_2^2$.
The convergence rate is exponential, since $\rho\in[0,1)$; therefore, $\E\left[\sum_{t=0}^\infty \psi_t \right]$ is bounded:
\begin{align*}
    &\sum_{t=0}^\infty \P_d^\pi(x_t\not\in\Xcal \land x_k\in\Xcal, \forall k\in\{0,...,t-1\})
    \\& \leq \sum_{t=0}^\infty  \frac{\E_d^\pi[\|x_t\|_2^2]}{a} 
    \\& \leq \frac{c}{a(1-\rho)} \|x_0\|_2^2
\end{align*}
and a sufficient condition such that \eqref{eq:cc_psi} is satisfied is $\frac{c}{a(1-\rho)} \|x_0\|_2^2 \leq  \veps$.

Mean-square stability can occur if, e.g., the disturbances are multiplicative, as in \cite{bernardini2012stabilizing}, or, with a similar argument, if the disturbances are additive with a decaying variance. In both cases, we do not need to assume that the uncertainty has a bounded support. Note, however, that the proposed conditions based on mean-square stability are neither more restrictive nor more general than assuming that the uncertainty has a bounded support, which is a common assumption in the infinite-horizon setting \cite{laurenti2025unifying}, but they suffice to show that boundedness of the support of the uncertainty is not a necessary condition to ensure feasibility of \eqref{cc_ocp}.

\subsection{Lagrange dual framework}\label{sec:duality}
We now consider the following problem, for a given $s_0\in\Scal$:
\begin{align}\label{cc_ocp_aug}
    \begin{split}
        V^\star(s_0) = \inf_{\pi\in\Gamma} \ & \E_d^\pi\left[ \sum_{t=0}^\infty \phi_t \ell(x_t, u_t) \right]
        \\ \text{s.t.} \ &  s_{t+1} =  F(s_t, u_t, d_t), \ \forall t\in\Z_{\geq0}
        \\& \E_d^\pi\left[\sum_{t=0}^\infty \psi_t \right] \leq \veps
        \\& u_t\sim\mu_t(\cdot \mid s_0,u_0,...,s_t), \ \forall t\in\Z_{\geq0}
        \\& \pi=(\mu_0, \mu_1, ...),
    \end{split}
\end{align}
where $V^\star:\Scal\to\R_{\geq0}$ is the value function.
Note that problem \eqref{cc_ocp_aug} is equivalent to \eqref{cc_ocp} for all $x_0\in\R^n$, $\xi_0=1, \psi_0=0$, and $\phi_0=1$. Indeed, in this case, the cost function equals $\sum_{t=0}^\infty\gamma^t\ell(x_t,u_t)$, and the constraint equals the one in \eqref{cc_ocp} in view of \eqref{eq:psi_constraint} and \eqref{eq:cc_psi}.
Since this is a constrained control problem, we emphasize the dependency of $V^\star$ on the initial state $s_0$. Then, $\Scal_\text{feas}:=\{s_0\in\Scal: s_0=[x_0^\top, 1, 0, 1]^\top, x_0\in\Xcal_\text{feas}\}$ denotes the set of initial conditions for which \eqref{cc_ocp_aug} is feasible. For problem \eqref{cc_ocp_aug}, the stochastic kernels $\{\mu_t\}_{t=0}^\infty$, are defined over $\Ucal\times(\Scal\times\Ucal)^{t-1}\times\Scal$.

Now, let us define the Lagrangian associated to \eqref{cc_ocp_aug}:
\begin{align}
L(s_0,\pi,\lambda) :=  
\E_d^\pi\Bigg[ 
        \sum_{t=0}^\infty \phi_t \ell(x_t, u_t) + \lambda\left(\sum_{t=0}^\infty \psi_t - \veps \right) \Bigg] \label{eq:lagrangian}
\end{align}
where, for compactness, we have implicitly substituted the system dynamics $s_{t+1}=F(s_t, u_t, d_t)$, and $u_t\sim\mu_t(\cdot|s_0,u_0 ..., s_t),\forall t \in\Z_{\geq0}$. Then, we consider the dual of \eqref{cc_ocp_aug}, for $s_0\in\Scal$:
\begin{align}\label{dual}
    D^\star(s_0):=  
    \sup_{\lambda\geq0} \inf_{\pi\in\Gamma} \ L(s_0,\pi,\lambda),
\end{align}
where the inner function $\inf_{\pi\in\Gamma} L(s_0,\pi,\lambda)$ is the \emph{dual function}.
The dual variable $\lambda$ is a scalar since, after substituting the system dynamics, the chance constraint is the only explicit constraint in the problem. Moreover, the optimal dual variable is, in general, a function $\lambda^\star:\Scal\to\R_{\geq0}$, as it depends on the initial state $s_0$. To see this, note that the constraint in \eqref{cc_ocp_aug} couples the entire trajectory starting from $s_0$. Hence, by explicitly substituting the system dynamics, the quantity $\E_d^\pi\left[\sum_{t=0}^\infty \psi_t \right]$ depends only on $s_0$ and the policy $\pi$, i.e., the decision variable. A different situation would be a control problem with an individual constraint for each state $s_t, t\in\Z_{\geq0}$; in this case, $\lambda$ would depend on the current state $s_t$.
From duality theory, weak duality always holds, i.e., $D^\star(s_0)\leq V^\star(s_0), \forall s_0\in\Scal$. Therefore, if the dual is unbounded from above, the primal is infeasible. Also, whenever $D^\star(s_0)= V^\star(s_0), \forall s_0\in\Scal_\text{feas}$, we say that strong duality holds for \eqref{cc_ocp_aug} and \eqref{dual}. 

Leveraging classical results from duality in constrained MDPs \cite{hernandez-lerma2000constrained, paternain2019constrained, schmid2025computing}, we can show a strong duality property for \eqref{cc_ocp_aug} and \eqref{dual}. This result is based on a well-known sufficient condition for strong duality that relies on the \emph{perturbation function}.
For any $s_0\in\Scal_\text{feas}$, define  $P_{s_0}:[-\veps,1-\veps]\to\R_{\geq0}$ 
\begin{align}\label{perturbed}
        \begin{split}
            P_{s_0}(\eta) = \inf_{\pi\in\Gamma} & \ \E_d^\pi\left[ \sum_{t=0}^\infty \phi_t \ell(x_t, u_t) \right]
            \\ \text{s.t.}& \ \E_d^\pi\left[\sum_{t=0}^\infty \psi_t \right] \leq \veps + \eta,
        \end{split}
    \end{align}
    where the (augmented) system dynamics are implicitly substituted again. As we show in the next proposition, the continuity of the perturbation function is strictly related to strong duality.
    
\begin{proposition}\label{prop:strong_duality}
    For all $s_0\in\Scal_\text{feas}$, assume that Slater's condition holds, i.e., there exists a policy $\pi\in\Gamma$ such that $P_{s_0}(\eta)<\infty$ for some $\eta>0$. Then, strong duality holds for \eqref{cc_ocp_aug} and \eqref{dual}, for all $s_0\in\Scal_\text{feas}$.
\end{proposition}
\begin{proof}
     A key step to prove strong duality is to show that $P_{s_0}$ is lower semicontinuous in $\eta=0$ \cite[Corollary 4.3.6]{borwein2006convex}. A sufficient condition is to show that $P_{s_0}$ is convex in an arbitrarily small set around $\eta=0$. Therefore, consider $\Acal:=\{\eta\in\R: \|\eta\|_2\leq\Bar{\eta}\}$, where $\Bar{\eta}$ is a positive constant such that $P_{s_0}(\eta)<\infty$ for all $\eta\in\Acal$, which is possible in view of Slater's condition. Consider any $\eta_1, \eta_2 \in \Acal$, with $\eta_1<0<\eta_2$. We now show that $P_{s_0}(\kappa\eta_1 + (1-\kappa)\eta_2) \leq \kappa P_{s_0}(\eta_1) + (1-\kappa)P_{s_0}(\eta_2)$, for any $\kappa\in[0,1]$. 
     Since the problem is feasible for $\eta\in\Acal$, the corresponding optimal values $P_{s_0}(\eta_1)$ and $P_{s_0}(\eta_2)$ exist and are bounded. Therefore, there exist sequences of minimizing feasible policies $\{\pi_1^i\}_{i=1}^{\infty}$, $\{\pi_2^i\}_{i=1}^{\infty}$ such that the following are satisfied:
    \begin{align}
        & \lim_{i\to\infty} \E_d^{\pi_1^i}\left[ \sum_{t=0}^\infty \phi_t \ell(x_t, u_t) \ \Big| \ s_0=s\right] = P_{s_0}(\eta_1) \label{eq:convergence1}
        \\& \lim_{i\to\infty} \E_d^{\pi_2^i}\left[ \sum_{t=0}^\infty \phi_t \ell(x_t, u_t) \ \Big| \ s_0=s\right] = P_{s_0}(\eta_2) \label{eq:convergence2}
        \\& \E_d^{\pi_1^i}\left[\sum_{t=0}^\infty 
        \psi_t \right] \leq \veps + \eta_1 , \forall i\in\Z_{\geq1} \label{eq:convergence3} 
        \\& \E_d^{\pi_2^i}\left[\sum_{t=0}^\infty  \psi_t \right] \leq \veps + \eta_2, \forall i\in\Z_{\geq1}. \label{eq:convergence4}
    \end{align}
    Then, we can construct a sequence of mixed policies $\{\Bar{\pi}^i\}$ such that $\Bar{\pi}^i$ selects $\pi_1^i$ with probability $\kappa$ and $\pi_2^i$ with probability $1-\kappa$. Hence, in view of the linearity of expectation and of \eqref{eq:convergence1}--\eqref{eq:convergence2}, the optimal value of $\Bar{\pi}^i$ converges to $\Bar{P}:=\kappa P_{s_0}(\eta_1) + (1-\kappa)P_{s_0}(\eta_2)$, for $i\to\infty$. For analogous reasons, the risk value of $\Bar{\pi}^i$ is not larger than $\veps + \kappa\eta_1 + (1-\kappa)\eta_2$, in view of \eqref{eq:convergence3}--\eqref{eq:convergence4}. Hence, $\Bar{\pi}^i$ is a sequence of feasible, although potentially suboptimal, policies for problem \eqref{perturbed} with $\eta=\kappa\eta_1 + (1-\kappa)\eta_2$. 
    Therefore, $P_{s_0}(\kappa\eta_1 + (1-\kappa)\eta_2)\leq \Bar{P}=\kappa P_{s_0}(\eta_1) + (1-\kappa)P_{s_0}(\eta_2),$ which implies that $P_{s_0}$ is convex in $\Acal$, hence continuous in its interior. In particular, $P_{s_0}$ is continuous for $\eta=0$, which implies strong duality for \eqref{cc_ocp_aug} and \eqref{dual}.
\end{proof}

Note that this proof is similar to the one in \cite{paternain2019constrained}, with the difference that we do not necessarily assume that the optimal value is attained in \eqref{perturbed}. This is indeed not necessary, since duality is a property that relates the optimal values of \eqref{cc_ocp_aug} and \eqref{dual}, independently of the existence of the minimizers.



In view of Proposition \ref{prop:strong_duality}, \eqref{cc_ocp_aug} is equivalent to \eqref{dual} for all $s_0\in\Scal_\text{feas}$; therefore, we can replace \eqref{cc_ocp_aug} by
\begin{align}\label{cc_ocp_dual}
\begin{split}
    V^\star(s_0)
    & = 
    \sup_{\lambda\geq0} \inf_{\pi\in\Gamma} L(s_0,\pi,\lambda)
    \\& = \sup_{\lambda\geq0} \inf_{\pi\in\Gamma} \E\Bigg[ 
        \sum_{t=0}^\infty (\phi_t \ell(x_t, u_t) + \lambda\psi_t) \Bigg] - \lambda\veps 
\end{split}
\end{align}
which is obtained from the definition of the Lagrangian $L$ in \eqref{eq:lagrangian} and by rearranging some terms.
A fundamental advantage of solving \eqref{cc_ocp_dual} instead of \eqref{cc_ocp_aug}, is that, for given $s_0\in\Scal_\text{feas}$ and $\lambda\geq0$, the inner minimization problem is a classical unconstrained model-based RL problem with the augmented stage cost
$$\phi\ell(x,u) + \lambda\psi,$$
which can be solved, e.g., via (approximate) dynamic programming \cite{sutton1998reinforcement}. In the next section, we detail our solution approach for \eqref{cc_ocp_dual}.


\section{Dual-ascent algorithm}\label{sec:dual_ascent}
In this section, we propose to solve \eqref{cc_ocp_dual} by means of a dual-ascent algorithm. Solving \eqref{cc_ocp_dual} via  dual ascent means maximizing the dual function
$\min_{\pi\in\Gamma} L(s_0,\pi,\lambda)$ with respect to $\lambda$; therefore, each update of the dual variable also involves optimizing the Lagrangian with respect to $\pi$, for fixed $\lambda$. In other words, \eqref{cc_ocp_dual} can be solved by updating the variables $\pi$ and $\lambda$ alternately, i.e., by updating one variable at a time and keeping the other one constant. Therefore, we first study the inner minimization problem, and establish convergence properties and the existence of a deterministic Markov policy that attains the optimal value; second, we prove that the proposed dual-ascent algorithm to solve \eqref{cc_ocp_dual} converges to an optimal and feasible deterministic policy for \eqref{cc_ocp_aug}.


\subsection{Inner minimization problem}\label{sec:inner_prob}
Let us now consider the solution of the inner minimization problem in \eqref{cc_ocp_dual} for a fixed $\lambda\in\R_{\geq0}$, which is instrumental for our primal-dual scheme. For this purpose, consider the function $J:\Scal\times\R_{\geq0}\to\R_{\geq0}$, defined as
\begin{align}\label{eq:primal_problem}
    J(s,\lambda):= \inf_{\pi\in\Gamma} \E_d^\pi\Bigg[ 
        \sum_{t=0}^\infty\left( \phi_t \ell(x_t, u_t) + \lambda \psi_t \right) \ \Big| \ s_0=s \Bigg].
\end{align}
This is an unconstrained control problem; therefore, we parametrize $J$ as a function of a generic state $s\in\Scal$. This is in contrast to $\lambda^\star$, which, as we have emphasized in Section \ref{sec:duality}, depends specifically on $s_0$.

Note that, for a given $s\in\Scal$, $J(s,\lambda) - \lambda\veps$ corresponds to the dual function, i.e., the objective of the maximization problem in \eqref{cc_ocp_dual}. However, for the purpose of solving the inner minimization problem with a fixed $\lambda\geq0$, we have dropped $\veps$ since it is constant with respect to $\pi$.
Then, \eqref{eq:primal_problem} corresponds to an unconstrained model-based RL problem.
A natural way to solve \eqref{eq:primal_problem} is via the following DP iterations \cite{bertsekas1996stochastic}:
\begin{align}
    & J_0(s,\lambda) = 0, \label{eq:dp1}
    \\& \begin{aligned}[t] \label{eq:dp2}
        J_{h+1}(s,\lambda) = \inf_{u\in\Ucal} \{
        & \phi\ell(x,u) + \lambda\psi 
        \\& + \E_d[J_h(F(s, u, d), \lambda) \ | \ s,u,\lambda]\}. 
    \end{aligned}
\end{align}
As shown later, under appropriate assumptions on the cost function and the constraint set, \eqref{eq:primal_problem} admits a stationary, deterministic, Markov optimal policy, and the iterations \eqref{eq:dp1}--\eqref{eq:dp2} converge to the optimal value function \cite{bertsekas1996stochastic}. However, for this to hold, the first step is to prove an important continuity property for the functions resulting from \eqref{eq:dp1}--\eqref{eq:dp2}. We state an additional assumption on the stage cost, which is common in several RL settings \cite{sutton1998reinforcement, bertsekas1996stochastic}:
\begin{assumption}\label{ass:bounded_cost}
    The stage cost is bounded in its domain, i.e., there exists an $\Bar{\ell}\geq0$ such that
    $\ell(x,u) \leq \Bar{\ell}, \  \forall (x,u)\in\R^{n+m}.$
\end{assumption}

In addition, we require an additional assumption needed to prove the continuity of $J_h$ in \eqref{eq:dp1}--\eqref{eq:dp2}:
\begin{assumption}\label{ass:0_measure}
    We assume that there exists a set $\Bar{\Xcal}\subseteq\Xcal$ such that 
    $
        \P(f(x,u,d) \in \partial \Xcal) = 0, \ \forall x\in\Bar{\Xcal}, \forall u\in\Ucal.
    $
\end{assumption}

This assumption is essentially needed to avoid possible discontinuities due to the indicator function in \eqref{eq:aug_dynamics}. Essentially, Assumption \ref{ass:0_measure} ensures that discontinuities are allowed over $\Bar{\Xcal}$ provided that they correspond to 0-measure events, which will be crucial to prove the continuity of $J_h$ in \eqref{eq:dp1}--\eqref{eq:dp2}. In Section \ref{sec:value_structure}, we will see that this assumption is verified in several cases, and we will give more insights on the computation of $\Bar{\Xcal}$.

Let us also define $\Bar{\Scal}:=\Bar{\Xcal}\times\{0,1\}^2\times[0,1]$. In the following proposition, we prove the continuity of $J_h$ defined in \eqref{eq:dp1}--\eqref{eq:dp2} over $\Bar{\Scal}\times\R_{\geq0}$.
\begin{proposition}\label{prop:continuity_iteration}
    Under Assumptions \ref{ass:cont}--\ref{ass:0_measure}, the functions $J_h, \forall h\in\Z_{\geq0}$, resulting from the DP iterations \eqref{eq:dp1}--\eqref{eq:dp2} are continuous over $\Bar{\Scal}\times\R_{\geq0}$.
\end{proposition}
\begin{proof}
    We proceed by induction to prove the continuity.
    The claim is true for $h=0$, since $J_0=0$. Let us now assume $J_h$ is continuous in its arguments. Then, we prove $J_{h+1}$ is continuous. Let us consider the function $G_h$ defined by
    $$G_h(s,u,\lambda):= \phi\ell(x,u) + \lambda\psi + \E_d\left[ J_h(F(s,u,d), \lambda) | s,u,\lambda \right]$$
    such that we have 
    \begin{align}\label{eq:G}
        J_{h+1}(s,\lambda) = \inf_{u\in\Ucal} G_h(s, u,  \lambda).
    \end{align}
    To prove the continuity of $J_{h+1}$, we first prove that $G_h$ is a continuous function, and then invoke Berge's Maximum Theorem \cite[Section 6.3]{berge1877topological} for parametric optimization. 
    First of all, we note that $G_h$ is the sum of three functions: since continuity is preserved under the sum, we focus on the three components separately. From Assumption \ref{ass:cont}, $\ell$, the stage cost, is a continuous mapping from $\R^n\times \R^m$ to $\R$. The term $\lambda\psi$ is trivially continuous in $\lambda$, for each $\psi\in\{0,1\}$. Now, let us analyze the third term; in particular, we study the continuity of 
    $\E_d\left[ J_h(F(s,u,d), \lambda) | s,u,\lambda \right]$. 
    Consider a sequence $\{(s_l,u_l,\lambda_l)\}_{l=0}^\infty$ such that $(s_l,u_l,\lambda_l) \to (s,u,\lambda)$, with $s_l\in\Bar{\Scal}, s\in\Bar{\Scal}, u\in\Ucal, \lambda_l\in\R_{\geq0}, \lambda\in\R_{\geq0}, \forall l\in\Z_{\geq0}$. The Continuous Mapping Theorem \cite[Theorem 3.2.10]{durrett2019probability} yields
    \begin{align}\label{eq:F_cont}
        F(s_l,u_l,d) \overset{\text{a.s.}}{\to} F(s,u,d), \ \forall d\in\Delta
    \end{align}
    since the dynamics of $x$ and $\phi$ are governed by a continuous function in view of Assumption \ref{ass:cont}, and the binary states converge up to zero-measure events in view of Assumption \ref{ass:0_measure}.  Also, the functions $\{J_h(s,\lambda)\}_{h=0}^\infty$ are bounded by an integrable (constant) function for all $s\in\R^n\times\{0,1\}^2\times[0,1]$ and $\lambda\in\R_{\geq0}$, since
    \begin{align*}
        J_h(s,\lambda) 
        & \leq \sup_{u_t\in\Ucal,\forall t}  \E_d^\pi \left[ \sum_{t=0}^{\infty} \gamma^t\ell(x_t, u_t)\right] + \lambda \sup_{\pi\in\Gamma}  \E_d^\pi \left[ \sum_{t=0}^{h-1}  \psi_t \right]
        \\& \leq \frac{\Bar{\ell}}{1-\gamma} + \lambda,
    \end{align*}
    where we have used known identities for the geometric series, and recalling that $\E_d^\pi \left[ \sum_{t=0}^h  \psi_t \right] \leq 1$ since it represents a probability in view of \eqref{eq:psi_constraint}. Hence, the Dominated Convergence Theorem applies, and together with \eqref{eq:F_cont}, it ensures that the expectation operator is a continuous function of $s,u$, and $\lambda$ (cf.\ \cite[Theorem 7.43]{shapiro2021lectures}.  Therefore, the function $G_h$ is continuous in $s,u$, and $\lambda$. Since $\Ucal$ is a compact set in view of Assumption \ref{ass:cont}, the infimum in \eqref{eq:G} is attained. 
    Then, Berge's Maximum Theorem \cite{berge1877topological} ensures that $J_{h+1}$, resulting from the parametric optimization problem \eqref{eq:G}, is continuous in its arguments.
\end{proof}
    
With this continuity result, we can prove the following proposition:
\begin{proposition}\label{prop:primal_convergence}
    Consider the DP iterations \eqref{eq:dp1}--\eqref{eq:dp2}. Under Assumptions \ref{ass:cont}--\ref{ass:0_measure}, $\lim_{h\to\infty} J_{h}(s,\lambda) = J(s,\lambda)$, $\forall s\in{\Bar{\Scal}}$. In particular, $J$ satisfies the following Bellman optimality equation:
    \begin{align}
        & J(s,\lambda) \label{eq:fixed_point}
        \\& = \min_{u\in\Ucal} \left\{ \phi \ell(x, u) + \lambda \psi + \E_d[J(F(s, u, d),\lambda) | s,u,\lambda] \right\}, \nonumber
    \end{align}
    and the optimal value in \eqref{eq:primal_problem} is attained by a deterministic, stationary, Markov policy, defined over $\Bar{\Scal}\times\R_{\geq0}$, computed by:
    \begin{align}\label{eq:policy}
        \pi(s, \lambda)\in\arg\min_{u\in\Ucal} \left\{ \phi \ell(x, u) + \lambda \psi + \E_d[J(F(s, u, d),\lambda)] \right\}.
    \end{align}
\end{proposition}
\begin{proof}
    The convergence of the DP iterations \eqref{eq:dp1}--\eqref{eq:dp2} follows from the fact that the sequence $\{J_h\}_{h\in\Z_{\geq0}}$ is monotonically increasing 
    \cite[Proposition 5.12]{bertsekas1996stochastic}. The fixed-point relation \eqref{eq:fixed_point} follows from \cite[Proposition 5.2]{bertsekas1996stochastic}. Finally, the continuity of $J_h$ from Proposition \ref{prop:continuity_iteration}
    implies that the set 
    $\{u\in\Ucal: \phi\ell(x,u) + \lambda\psi + \E_d[J_h(F(s,u,d),\lambda) | s,u,\lambda]\leq z\}$
    is compact, for all $s\in\Bar{\Scal}$. In addition, since \eqref{eq:primal_problem} is a stationary optimal control problem, \cite[Proposition 5.10]{bertsekas1996stochastic} ensures that the optimal value of \eqref{eq:primal_problem} is attained by a stationary deterministic Markov policy over $\Bar{\Scal}\times\R_{\geq0}$.
\end{proof}

Note that Propositions \ref{prop:continuity_iteration} and \ref{prop:primal_convergence} allow to replace the infimum in \eqref{eq:primal_problem} and \eqref{eq:dp2} by the corresponding minimum. In addition, they demonstrate that the function $J$, i.e., the value function with penalty $\lambda$ on the chance constraint, can be learned by means of unconstrained RL algorithms (e.g., value iteration), via the iterations \eqref{eq:dp1}--\eqref{eq:dp2}.

\subsection{Dual-ascent algorithm design}
In view of Propositions \ref{prop:continuity_iteration} and \ref{prop:primal_convergence}, we can now design the dual ascent-steps to solve \eqref{cc_ocp_dual}. 
For convenience, we rewrite \eqref{eq:lagrangian} as
\begin{align}\label{eq:lagrangian_2a}
    L(s_0,\pi,\lambda) := V^\pi_\text{cost}(s_0) + \lambda (V^\pi_\text{risk}(s_0) - \veps),
\end{align}
with
\begin{align}\label{eq:lagrangian_2b}
\begin{split}
    &V^\pi_\text{cost}(s_0) = \E_d^\pi\left[ 
        \sum_{t=0}^\infty \phi_t \ell(x_t, u_t) \right], 
    \\& V^\pi_\text{risk}(s_0) = \E_d^\pi\left[\sum_{t=0}^\infty \psi_t \right],
\end{split}
\end{align}
where $V^\pi_\text{cost}$ is the component of the value function that measures the system performance under the policy $\pi$ and initial state $s_0$, whereas, in view of \eqref{eq:psi_constraint}, $V^\pi_\text{risk}$ equals the probability that the system violates the safety constraint, under policy $\pi$ and initial state $s_0$.

By initializing $\lambda_0(s_0)=0, \forall s_0\in\Scal$ and $k=0$, the following two steps are performed sequentially:
\begin{itemize}
    \item Given $\lambda_k(s_0)$, update the policy via a model-based RL algorithm $\forall  s_0\in\Scal$:
    \begin{align}\label{eq:primal_update}
    \pi_{k+1} &\in \arg\min_{\pi \in \Gamma}\mathbb{E}_d^\pi[L(s_0, \pi, \lambda_k(s_0))]
    \\ &\in \arg\min_{\pi \in \Gamma}\E_d^\pi\Bigg[ 
        \sum_{t=0}^\infty\left( \phi_t \ell(x_t, u_t) + \lambda_k(s_0) \psi_t \right) \Bigg]\nonumber
    \end{align}
    which is solved from the iterations \eqref{eq:dp1}--\eqref{eq:dp2} with $\lambda=\lambda_k(s_0)$, and a deterministic Markov policy exists in view of Proposition \ref{prop:primal_convergence};
    \item Update the dual variable via projected dual ascent, $\forall s_0\in\Scal$:
    \begin{align}\label{eq:dual_update}
    \lambda_{k+1}(s_0) = \left[\lambda_k(s_0) + \alpha_k (V_\text{risk}^{\pi_{k+1}}(s_0)-\varepsilon)\right]_{\geq0}
    \end{align}
    where $\{\alpha_k\}_{k=0}^\infty$ is a decaying learning rate, which will be designed later.
\end{itemize}
The rationale behind \eqref{eq:primal_update}--\eqref{eq:dual_update} is as follows: for a fixed $\lambda$, the only decision variable in \eqref{cc_ocp_dual} is $\pi$; hence, the inner minimization problem in \eqref{cc_ocp_dual} can be solved via unconstrained model-based RL approaches \cite{sutton1998reinforcement}, and a minimizer policy exists in view of Proposition \ref{prop:primal_convergence}. Then, for a fixed policy $\pi$, the outer maximization problem in \eqref{cc_ocp_aug} is linear in $\lambda$; hence, in \eqref{eq:dual_update} we update the dual variable by taking a feasible step in the direction of the steepest increase, since, in view of \eqref{eq:lagrangian_2a}--\eqref{eq:lagrangian_2b}, we have
$\nabla_\lambda L(s_0,\pi,\lambda) = V_\text{risk}^\pi(s_0)-\veps.$
Then, at each iteration, the dual variable is increased proportionally to the amount by which the risk value exceeds $\veps$, thereby increasing the penalty on constraint violations. Note that the dual variable grows unbounded for points $s_0\in\Scal$ that are infeasible, i.e., those points for which all input sequences give a risk larger than $\veps$.

\subsection{Convergence analysis}
Whenever the step sizes $\{\alpha_k\}_{k=0}^\infty$ satisfy standard conditions 
\begin{align}\label{eq:alpha_ass}
    \sum_{k=0}^\infty \alpha_k = \infty, \quad \sum_{k=0}^\infty \alpha_k^2 < \infty, \quad
    \alpha_k \geq0, \ \forall k\in\Z_{\geq0},
\end{align}
the dual-ascent steps \eqref{eq:primal_update}--\eqref{eq:dual_update} converge to the optimal dual variable $\lambda^\star(s_0)$ of \eqref{cc_ocp_dual}, $\forall s_0\in\Scal$ \cite[Exercise 6.3.13]{bertsekas1997nonlinear}. 
Then, if the dual variable is bounded, the corresponding optimal policy of \eqref{cc_ocp_aug} can be extracted from 
\begin{align}\label{eq:optimal_policy}
\begin{split}
    \pi^\star(s) \in \arg\min_{u\in\Ucal} \{& \phi\ell(x, u) + \lambda^\star(s_0) \psi 
    \\& + \E_d[J(F(s, u, d),\lambda^\star(s_0))] \},
\end{split}
\end{align}
or from the limit for $k\to\infty$ of \eqref{eq:primal_update}--\eqref{eq:dual_update}.

If the minimizer of \eqref{cc_ocp_aug} were unique, we would then be able to conclude that $\pi^\star(s)$ in \eqref{eq:optimal_policy} is both optimal and feasible for \eqref{cc_ocp_aug}. However, this is generally not the case, as there may exist multiple primal solutions that minimize the Lagrangian given the optimal multiplier, i.e., policies $\pi^\star$ found from \eqref{eq:optimal_policy}.
They all achieve the same optimal value, but they are not necessarily primal feasible \cite[Section 6]{bertsekas1997nonlinear}. Therefore, the feasibility of the policy resulting from \eqref{eq:optimal_policy} for \eqref{cc_ocp_aug} is not obvious. Moreover, \cite{schmid2025computing} proves feasibility for mixed policies only, and \cite{chen2024probabilistic} proves feasibility on average over the learning iterations.
The following proposition is a key contribution regarding feasibility, as it shows that feasibility can be achieved also by the \emph{deterministic} policy resulting from \eqref{eq:optimal_policy}, or from the limit of the iterations \eqref{eq:primal_update}--\eqref{eq:dual_update}.

\begin{proposition}\label{prop:feasibility}
Assume that the step sizes $\{\alpha_k\}_{k=0}^\infty$ satisfy \eqref{eq:alpha_ass}.
Then, under Assumption \ref{ass:cont}, for any $s_0\in\Scal$ such that $\lambda^\star(s_0)<\infty$, the policy $\pi^\star$ in \eqref{eq:optimal_policy} resulting from the dual ascent \eqref{eq:primal_update}--\eqref{eq:dual_update} is feasible for problem \eqref{cc_ocp_aug}.
\end{proposition}
\begin{proof}
Note that \eqref{eq:alpha_ass} ensures the convergence of the dual ascent steps \cite[Exercise 6.3.13]{bertsekas1997nonlinear}: $\lim_{k\to\infty}\lambda_k(s_0)=\lambda^\star(s_0)$. By considering $s_0\in\Scal$ such that $\lambda^\star(s_0)<\infty$, from \eqref{eq:dual_update} we have
\begin{align*}
    \lambda_{k+1}(s_0) 
    &\geq \lambda_{k}(s_0) + \alpha_k (V_\text{risk}^{\pi_{k+1}}(s_0)-\varepsilon)
    \\& \geq \lambda_0(s_0) +\sum_{j=0}^{k}  \alpha_j (V_\text{risk}^{\pi_{j+1}}(s_0)-\varepsilon),
\end{align*}
where we have used the non-expansiveness of the projection operator, and we have iterated backwards from iteration $k$ to iteration 0. By taking $k\to \infty$ and dividing by $\sum_{j=0}^{\infty} \alpha_j$, we have
\begin{align}\label{eq:safety_2}
\begin{split}
    \frac{\sum_{j=0}^{\infty} \alpha_j (V_\text{risk}^{\pi_{j+1}}(s_0)-\varepsilon)}{\sum_{j=0}^{\infty} \alpha_j } \ 
    &\leq \ \frac{\lambda^\star(s_0) - \lambda_0(s_0)}{\sum_{j=0}^\infty \alpha_j}
    \\& \leq 0,
\end{split}
\end{align}
where the second inequality follows from \eqref{eq:alpha_ass}, and since we consider $s_0\in\Scal$ such that $\lambda^\star(s_0)$ is finite.
Now, we invoke the Stolz-Cesàro Theorem \cite[Section 3.1.7]{muresan2009concrete}, which states that, considering 
two sequences $\{a_{k}\}_{k\geq 1}$ and $\{b_{k}\}_{k\geq 1}$ of real numbers, with $\{b_{k}\}_{k\geq 1}$ strictly monotone and divergent, and assuming that the limit
$$ \lim_{k\rightarrow\infty} \frac{a_{k+1} - a_k}{b_{k+1} - b_k} = l$$
exists, then $$ \lim_{k\rightarrow\infty} \frac{a_k}{b_k} = l.$$
In our case, we can simply set $a_k = \sum_{j=0}^k \alpha_j (V_\text{risk}^{\pi_{j+1}}(s_0)-\varepsilon)$ and $b_k = \sum_{j=0}^k \alpha_j$. Note that it holds that $a_{k+1} = a_k + \alpha_{k+1} (V_\text{risk}^{\pi_{k+1}}(s_0)-\varepsilon)$, and $b_{k+1} = b_k + \alpha_{k+1}$. Hence, in particular, $b_k$ is monotone and divergent in view of \eqref{eq:alpha_ass}. Then we have:
\begin{align*}
    \lim_{k\rightarrow\infty} \frac{a_{k+1} - a_k}{b_{k+1} - b_k} 
    & = \lim_{k\rightarrow\infty} \frac{(a_k + \alpha_{k+1} (V_\text{risk}^{\pi_{k+1}}(s_0)-\varepsilon)) - a_k}{(b_k + \alpha_{k+1}) - b_k}
    \\& = \lim_{k\rightarrow\infty} \frac{\alpha_{k+1} (V_\text{risk}^{\pi_{k+1}}(s_0)-\varepsilon)}{\alpha_{k+1}}
    \\& = V_\text{risk}^{\pi^\star}(s_0)-\varepsilon.
\end{align*}
The Stolz-Cesàro Theorem then yields that
\begin{align*}
    V_\text{risk}^{\pi^\star}(s_0)-\varepsilon 
    & = \lim_{k\rightarrow\infty} \frac{a_k}{b_k} 
    \\& = \lim_{k\rightarrow\infty}\frac{\sum_{j=0}^{\infty} \alpha_j(V_\text{risk}^{\pi_{j+1}}(s_0)-\varepsilon)}{\sum_{j=0}^{\infty} \alpha_j } 
    \\& \leq 0,
\end{align*}
where the first equality is the Stolz-Cesàro Theorem, and the last inequality follows from \eqref{eq:safety_2}. This shows that $V_\text{risk}^{\pi^\star}(s_0)\leq \varepsilon.$ 
\end{proof}

Note that, since $\pi^\star$ in \eqref{eq:optimal_policy} is shown to be feasible for \eqref{cc_ocp_aug}, i.e., the problem defined in the augmented state space, the state component $x_t, t\in\Z_{\geq0}$, of the closed-loop system controlled by $\pi^\star$  will meet the probabilistic constraint in \eqref{cc_ocp}, in view of the equivalence between \eqref{cc_ocp} and \eqref{cc_ocp_aug}.
Also, the proposition above allows to characterize the feasible set $\Scal_\text{feas}$ as $\{s_0\in\Scal: \lambda^\star(s_0)<\infty\}$.

\section{Learning algorithm}\label{sec:learning}
The iterations \eqref{eq:primal_update}--\eqref{eq:dual_update} have to be performed for all $s\in\Scal$, and since we assume continuous state and action spaces, they can be intractable.
Secondly, \eqref{eq:primal_update}--\eqref{eq:dual_update} require to solve an unconstrained optimal control problem for each update of the dual variable, which can be computationally expensive even in an offline training.
For this reason, in this section we propose to directly approximate the function $J$ in \eqref{eq:primal_problem} by means of a neural network (NN), and to perform the updates \eqref{eq:dp1}--\eqref{eq:dp2} for a batch of data points sampled from $\Scal$, and for a set of values for $\lambda$ in $[0, \lambda_\text{max}]$, where $\lambda_\text{max}$ is a large enough constant. 

\subsection{Continuity of the value function}
A fundamental step before training an NN is to investigate whether the function of interest is continuous. Indeed, NNs with classical activation functions (e.g., ReLU or hyperbolic tangent) are inherently continuous, and the Universal Approximation Theorem \cite{hornik1989multilayer} ensures that NNs can learn any continuous function with arbitrary accuracy, provided that the architecture is sufficiently expressive. 

Thus, the first step is to assess the continuity of $J$ defined in \eqref{eq:primal_problem}. Note that, even though the functions $J_h$ are continuous for each $h\in\Z_{\geq0}$ in view of Proposition \ref{prop:continuity_iteration}, this does not necessarily imply the continuity of the limit function $J$. To show this, a key step is to prove that the sequence $\{J_h\}_{h=1}^\infty$ converges \emph{uniformly} to $J$.

\begin{proposition}\label{prop:cont_value}
    Under Assumptions \ref{ass:cont}--\ref{ass:0_measure}, the function $J$ defined as in \eqref{eq:primal_problem} is continuous over $\Bar{\Scal}\times[0, \lambda_\text{max}]$.
\end{proposition}
\begin{proof}
From Proposition \ref{prop:continuity_iteration}, we know that $J_h$ is a continuous function $\forall h\in\Z_{\geq0}$ over $\Bar{\Scal}\times\R_{\geq0}$. To show that $J$ is continuous, we need to show that the sequence $\{J_h\}_{h=0}^\infty$ resulting from \eqref{eq:dp1}--\eqref{eq:dp2} converges uniformly to $J$.
Uniform convergence is equivalent to the following uniform bound on the tail of the infinite-horizon problem \cite[Theorem 7.9]{rudin2021principles}:
\begin{align}\label{eq:uniform_condition}
    \lim_{h\to\infty} \tau_h = 0
\end{align}
with 
\begin{align}\label{eq:uniform_condition2}
    \tau_h = \sup_{s\in\Bar{\Scal}, \lambda\in[0, \lambda_\text{max}]} \left| J(s,\lambda) - J_h(s,\lambda) \right|. 
\end{align}
Note that in view of the monotonicity property of the value functions $J_h$ (see the proof of Proposition \ref{prop:primal_convergence}), the absolute value in \eqref{eq:uniform_condition2} can be removed. For a given policy $\pi\in\Gamma$, let us define $J_h^\pi(s,\lambda) := \E_d^\pi\left[\sum_{t=0}^{h-1} \left(\phi_t\ell(x_t, u_t) + \lambda \psi_t \right) | s_0 = s \right]$ for $h\in\Z_{\geq1}$, and  $J^\pi_\infty(s,\lambda) := \lim_{h\to\infty} J_h^\pi(s,\lambda)$. Then, we have
\begin{align*}
    \tau_h & = \sup_{s\in\Bar{\Scal}, \lambda\in[0, \lambda_\text{max}]}  \left(\min_{\pi\in\Gamma}  J_\infty^\pi(s,\lambda) - \min_{\pi\in\Gamma} J_h^\pi(s,\lambda) \right)
    \\& \leq  \sup_{s\in\Bar{\Scal},\lambda\in[0, \lambda_\text{max}]}  \sup_{\pi\in\Gamma} \left(  J_\infty^\pi(s,\lambda) - J_h^\pi(s,\lambda) \right)
    \\& \leq  \sup_{s\in\Bar{\Scal}, \lambda\in[0, \lambda_\text{max}]}  \sup_{\pi\in\Gamma} \E_d^\pi\left[\sum_{t=h}^{\infty} \left(\phi_t \ell(x_t, u_t) + \lambda \psi_t \right) \mid s_0 = s \right]
    \\& \leq 
        \begin{aligned}[t]
        & \sup_{s\in\Bar{\Scal}} \sup_{u_t\in\Ucal, \forall t} \E_d^\pi\left[\sum_{t=h}^{\infty} \gamma^t \ell(x_t, u_t)\right]  
        \\& +  \lambda_\text{max}  \sup_{u_t\in\Ucal, \forall t} \sup_{\psi_h\in\{0,1\}}   \E_d^\pi\left[\sum_{t=h}^{\infty} \psi_t \right]
    \end{aligned}
    \\& \leq \frac{\gamma^h}{1-\gamma} \Bar{\ell} + \lambda_\text{max}  \sup_{u_t\in\Ucal} \sup_{\xi_h\in\{0,1\}} \E_d^\pi\left[ \xi_h - \lim_{t\to\infty} \xi_t \right],
\end{align*}
where we have used known identities for the geometric series, the boundedness of the stage cost in Assumption \ref{ass:bounded_cost}, and the dynamics of the augmented system \eqref{eq:aug_dynamics}. In particular, this shows that 
$\lim_{h\to\infty} \tau_h = 0,$ since 
\begin{align*}
    0 & \leq \tau_h
    \\& \leq \frac{\gamma^h}{1-\gamma} \Bar{\ell} + \lambda_\text{max}  \sup_{u_t\in\Ucal} \sup_{\xi_h\in\{0,1\}} \E_d^\pi\left[ \xi_h - \lim_{t\to\infty} \xi_t \right]
    \underset{h\to\infty}{\longrightarrow} 0,
\end{align*}
where we have used that $\lim_{h\to\infty} \E_d^\pi\left[\xi_h\right] = \E_d^\pi\left[\lim_{h\to\infty} \xi_h\right] $, again in view of the Dominated Convergence Theorem since $\{\xi_h\}_{h=0}^\infty$ is a bounded and convergent sequence \cite[Theorem 7.43]{shapiro2021lectures}. 
\end{proof}

Note that, although not specifically required next, the application of Berge's Maximum Theorem \cite[Section 6.3]{berge1877topological} on the continuous function $J$ also implies the continuity of the value function $V^\star$ over $\Scal_\text{feas}$.


\subsection{Structure of the value function}\label{sec:value_structure}

Before training a neural network, it can be beneficial to investigate structural properties of the value function of interest, especially considering that the state space $\Scal$ contains both continuous and discrete variables.

First, we focus on the dynamics of the augmented system \eqref{eq:aug_dynamics}. In particular, only the following transitions $(\xi_t, \psi_t)\to(\xi_{t+1}, \psi_{t+1})$ are possible, for $t\in\Z_{\geq0}$:
\begin{align}
    &(1,0)\to(1,0), \quad (1,0)\to(0,1), \label{eq:bin_trans1}
    \\& (0,1)\to(0,0), \quad (0,0)\to(0,0).\label{eq:bin_trans2}
\end{align}
The first transition in \eqref{eq:bin_trans1} indicates that all states $x_0,....,x_{t}, x_{t+1}$ are in $\Xcal$, whereas the second transition in \eqref{eq:bin_trans1} indicates that $x_{t+1}$ is the first state that violates the constraint. Then, the first transition in \eqref{eq:bin_trans2} holds because, once the first violation has occurred, both binary states are reset to 0 in view of \eqref{eq:aug_dynamics}, and the second transition in \eqref{eq:bin_trans2} similarly follows. In particular, note that, in view of \eqref{eq:aug_dynamics}, only three combinations are possible for $(\xi,\psi)\in\{0,1\}^2$, out of the four possible ones (i.e., (1, 0), (0, 1), and (0, 0)), and only four transitions $(\xi_t, \psi_t)\to(\xi_{t+1}, \psi_{t+1})$ are possible, out of the sixteen possible ones.

In view of this observation, we can rewrite the function $J$ in a more convenient way. First, note that, for a given state $s\in\Scal$ such that $s=[x,1,0,\phi]^\top$, after a transition of the type $(1,0)\to(0,1)$, the states $\xi$ and $\psi$ will be identically 0 in view of \eqref{eq:bin_trans2}, for any uncertainty realization. Therefore, the safety certification expressed by the state $\psi$ in \eqref{eq:cc_psi} is disregarded as soon as the first constraint violation has occurred. This is not surprising, since in \eqref{eq:psi_constraint}, \eqref{eq:cc_psi} we relate the joint-in-time chance constraint to the probability that a certain state $x_t$, for some $t\in\Z_{\geq0}$, is the first one that violates the constraint, which is indeed embedded in the corresponding state $\psi_t$. Therefore, suppose that, for a certain uncertainty realization, we observe the first constraint violation at time step $t\in\Z_{\geq1}$. In this case, $\psi_t$ becomes 1, and the optimal behavior of the system thereafter is to disregard the safety constraint and optimize performance only. Therefore, 
let $W^\star:\R^n\times[0,1]\to\R_{\geq0}$ be the value of the optimal control problem that minimizes the unconstrained performance, i.e.:
\begin{align}\label{eq:W}
    W^\star(x,\phi) = \min_{\pi\in\Gamma} \E_x^\pi \left[ \sum_{t=0}^\infty \phi_t \ell(x_t, u_t) \ \Big| \ x_0=x, \phi_0=\phi \right],
\end{align}
which is a function of $x$ and $\phi$ only.
In view of the previous consideration, we can express the function $J$ in \eqref{eq:primal_problem} and the Bellman equation \eqref{eq:fixed_point} as 
\begin{align}\label{eq:J_explicit}
    &J(s,\lambda)\nonumber
    \\& = \begin{cases}
        \begin{aligned}[t]
            \min_{u\in\Ucal} \Bigl\{
                \phi\ell(x,u)
                + \E_d\!\left[J(F(s,u,d),\lambda)\mid s,u\right]
            \Bigr\}
            \\ \text{if } s\in\Scal,\; s=[x,1,0,\phi],
            \end{aligned}
            \\[1ex]
            W^\star(x,\phi) + \lambda\psi
             \qquad \text{otherwise}.
        \end{cases} 
    \end{align}
Indeed, in the first line of \eqref{eq:J_explicit}, we consider the case in which a violation has not yet occurred, i.e., $(\xi,\psi)=(1,0)$. In the second line, we consider $\xi=0$ and $\psi\in\{0,1\}$, i.e., when a violation has (just) occurred, and only performance optimization is retained. Essentially, in \eqref{eq:J_explicit}, we have used that $J(s,\lambda)$ coincides with $W^\star(x,\phi)+\lambda\psi$, for any state $s\in\Scal$ such that $\xi=0$.
Note that the term $\lambda\psi$ is still present in the second line of \eqref{eq:J_explicit}. Indeed, if $\xi=0$, $\psi$ can be either 0 or 1 in view of \eqref{eq:bin_trans2}, and it would be identically 0 thereafter. For this reason, only the term $\lambda\psi$, associated with the current state $s$, is present in the second line \eqref{eq:J_explicit}.

\subsection{Learning algorithm and practical implementation}\label{sec:practical_implementation}

For practical implementation, we propose to learn the function $W^\star$ in a separate training process, since the second line in \eqref{eq:J_explicit} depends exclusively on $W^\star$. This is a standard unconstrained value-iteration problem that employs an NN as a function approximator, and it is summarized in Algorithm \ref{alg:learn_W}, which outputs an NN $\widehat{W}:\Xcal\times[0,1]\to\R_{\geq0}$. In particular, the value function $ W^\star$ satisfies the following Bellman optimality equation, $\forall x\in\R^n, \forall\phi\in[0,1]$:
\begin{align*}
    W^\star(x,\phi) = \min_{u\in\Ucal} \big\{
    & \phi\ell(x,u) 
    \\& + \E_d\left[W^\star(f(x,u,d), \gamma\phi)\mid x,\phi,u\right]
\big\},
\end{align*}
which is then used in iterations of Algorithm \ref{alg:learn_W}.
As constraint violations are allowed, $W^\star$ may be evaluated in a state $x$ that is outside the safe set $\Xcal$. Therefore, in some applications, it can be advisable to learn $W^\star$ over the entire state space, or in a sufficiently larger superset of $\Xcal_\text{sup}\supseteq\Xcal$. Note that convergence guarantees are still preserved, since, in view of Assumption \ref{ass:bounded_cost}, the stage cost is bounded over the state-action space.

Then, the remaining learning problem is to approximate $J$ for $s\in\Scal$ such that $s=[x,1,0,\phi]^\top$, i.e., the first line of \eqref{eq:J_explicit}. This is summarized in Algorithm \ref{alg:learn_J}, which outputs an NN $\widehat{J}:\Xcal\times[0,1]\times[0,\lambda_\text{max}]\to\R_{\geq0}$ since the dependency is only on $x,\phi$, and $\lambda$, given that we consider $\xi=1$ and $\psi=0$. The idea is conceptually the same as the approximate value-iteration approach in Algorithm \ref{alg:learn_W}, but we employ the special structure for the one-step-ahead target computation, by following \eqref{eq:J_explicit} to evaluate $\widehat{J}$ at the next state. Specifically, in Step \ref{state:targetJ} of Algorithm \ref{alg:learn_J}, we use the law of conditional expectation to distinguish whether the next state belongs to $\Xcal$ or not, and we use the related expression for $J$ according to \eqref{eq:J_explicit} and the iterations \eqref{eq:dp1}--\eqref{eq:dp2}. In practice, the probability that the next state belongs to $\Xcal$ for given $s$ and $u$ can be approximated, e.g., empirically via counting, and the same argument applies to the related expectation operators. Last, note that, since we learn $J$ considering $(\xi, \psi)=(1,0)$, the next state $\psi^+$ is necessarily 1 if $x^+\not\in\Xcal$ in Step \ref{state:targetJ}.

Finally, to run Algorithm \ref{alg:learn_J}, we need to provide a set $\Bar{\Xcal}$ that satisfies Assumption \ref{ass:0_measure}. This assumption requires that the boundary of the safe set $\Xcal$ can be reached only in 0-measure events. This is the case if the distribution of $x_t, t\in\Z_{\geq0}$, admits a density and the boundary of $\Xcal$ has measure 0 in $\R^n$. Then, in this case, we have that $x_t\not\in\partial\Xcal$ almost surely, $\forall t\in\Z_{\geq0}$. Therefore, when the disturbances have a continuous distribution and $\partial\Xcal$ has measure 0 in $\R^n$, the only issue that can occur is that the disturbances of the system do not affect all the components of the state vector. This is because some components of the state would then evolve deterministically, and therefore their distribution would not admit a density. Hence, for such state components, the chance constraint is equivalent to a hard constraint. To compute $\Bar{\Xcal}$, consider the following illustrative example:
$$\begin{cases}
    x_{t+1}^{(1)} = x_{t}^{(1)} + x_{t}^{(2)}
    \\ x_{t+1}^{(2)} = f(x_t,u_t,d_t),
\end{cases}$$
with constraint $\P(x_{t}\in[-2,2]^2, \forall t\in\Z_{\geq0}) \geq 1-\veps$. Let $f$ be a continuous function, and let the distribution of $d$ admit a density. It is clear that $x_{t+1}^{(1)} = x_{t}^{(1)} + x_{t}^{(2)} = \pm 2$ is not a 0-measure event, since the dynamics equation for $x_{t}^{(1)}$ is deterministic. In accordance with Assumption \ref{ass:0_measure},  $\Bar{\Xcal}$ can be then chosen as $\Bar{\Xcal} = \{x\in\R^2: x_1 + x_2 \in (-2,2)\}$.
This avoids possible discontinuities introduced by disturbance-free components.
Note that it might still happen that states in $\Bar{\Xcal}$ yield a probability of constraint violation greater than $\veps$. The actual feasible set $\Scal_\text{feas}$ is then determined when running the dual ascent algorithm \eqref{eq:dual_update}, and consists of all points $x\in\Xcal$ for which the dual variable is finite in view of Proposition \ref{prop:feasibility}, or, in practice, lower than the upper bound $\lambda_\text{max}$.

\begin{algorithm}[t]
\caption{Learning $W^\star$}
\label{alg:learn_W}
\begin{algorithmic}[1]
\Require A superset $\Xcal_\text{sup}\supseteq \Xcal$, a data set $\Dcal_W := \{(x^i, \phi^i)\}_{i=1}^{N_{\textup{samples}, W}} \subseteq (\Xcal_\text{sup}\times [0,1])^{N_{\textup{samples}, W}}$, iterations limit $N_\text{iter}\in\Z_{>0}$, and  neural networks $\{\widehat{W}_{\theta, j}\}_{j=0}^{N_\text{iter}}$ such that $\widehat{W}_{\theta, j}: \Xcal_\text{sup}\times [0,1] \to \R, \forall j\in\{0, ..., N_\text{iter}\}$.
\State Initialize: $\widehat{W}_{\theta, 0}(x,\phi) \leftarrow 0, \quad \forall (x,\phi)\in\Xcal_\text{sup}\times [0,1]$.
\For{$j=1,...,N_\text{iter}$}
    \State Compute target $\forall (x,\phi)\in\Dcal_W$: 
    \begin{align*}
        T_{j}(x,\phi) \leftarrow &\min_{u\in\Ucal} \big\{
            \phi\ell(x,u) 
            \\& + \E_d[\widehat{W}_{\theta, j-1}(f(x,u,d),\gamma\phi) \mid x,\phi,u] \big\}, 
    \end{align*}
    \State Train NN:
        \begin{align*}
            \theta^\star \leftarrow \arg\min_{\theta} \Big\{ &\frac{1}{N_{\textup{samples}, W}} \sum_{i=1}^{N_{\textup{samples}, W}} \Big\| \widehat{W}_{\theta}(x^i,\phi^i) \\& - T_j(x^i, \phi^i)  \Big\|_2^2 \Big\} ,
        \end{align*}
    \State Set: $\widehat{W}_{\theta, j} \leftarrow \widehat{W}_{\theta^\star, j}$
\EndFor
\\
\Return $\{\widehat{W}_{\theta, j}\}_{j=1}^{N_\text{iter}}$.
\end{algorithmic}
\end{algorithm}

\begin{algorithm}[t]
\caption{Learning $J^\star$}
\label{alg:learn_J}
\begin{algorithmic}[1]
\Require A data set $\Dcal_J := \{(x^i, \phi^i, \lambda^i)\}_{i=1}^{N_{\textup{samples}, J}} \subseteq (\Bar{\Xcal}\times [0,1]\times{[0, \lambda_\text{max}]})^{N_{\textup{samples}, J}}$, iterations limit $N_\text{iter}\in\Z_{>0}$, and  neural networks $\{\widehat{J}_{\theta, j}\}_{j=0}^{N_\text{iter}}$ such that $\widehat{J}_{\theta, j}: \Xcal\times[0,1] \times [0, \lambda_\text{max}] \to \R, \forall j\in\{0, ..., N_\text{iter}\}$.
\State Initialize: $\widehat{J}_{\theta, 0}(x,\phi,\lambda) \leftarrow 0, \quad \forall (x,\phi, \lambda)\in\Xcal\times[0,1]\times[0, \lambda_\text{max}]$.
\For{$j=1,...,N_\text{iter}$}
    \State\label{state:targetJ} Compute target $\forall (x,\phi, \lambda)\in\Dcal_J$: 
    \begin{align*}
        & g_j(x,\phi,\lambda, u)
        \leftarrow \phi\ell(x,u) 
             \\& \begin{aligned}[t] 
             + \E_d\big[&\widehat{J}_{\theta, j-1}(f(x,u,d),\gamma\phi,\lambda) \mid 
              x,u, \phi, 
              \\&  f(x,u,d)\in\Xcal\big] 
             \cdot \P(f(x,u,d)\in\Xcal \mid x, u) &
             \end{aligned}
             \\& \begin{aligned}[t]
                 + \Big(\E_d\big[&\widehat{W}_{\theta, j-1}(f(x,u,d),\gamma\phi) \mid  x,u,  \phi, 
                 \\& f(x,u,d)\not\in\Xcal\big] +\lambda\Big)
                  \cdot \P(f(x,u,d)\not\in\Xcal \mid x, u) 
             \end{aligned}
        \\& T_j(x,\phi, \lambda) \leftarrow \min_{u\in\Ucal} g_j(x,\phi,\lambda,u)
    \end{align*}
    \State Train NN:
        \begin{align*}
            \theta^\star \leftarrow \arg\min_{\theta} \Big\{ &\frac{1}{N_{\textup{samples}, J}} \sum_{i=1}^{N_{\textup{samples}, J}} \Big\| \widehat{J}_{\theta}(x^i,\phi^i, \lambda_i) \\& - T_j(x^i, \phi^i, \lambda_i)  \Big\|_2^2 \Big\} ,
        \end{align*}
    \State Set: $\widehat{J}_{\theta, j} \leftarrow \widehat{J}_{\theta^\star, j}$
\EndFor
\\
\Return $\widehat{J}_{\theta, N_\text{iter}}$.
\end{algorithmic}
\end{algorithm}

Once an NN approximating $J$ is obtained by means of Algorithm \ref{alg:learn_J}, 
we can construct an approximated version of \eqref{eq:J_explicit}, for $s\in\Scal, \lambda\in[0,\lambda_\text{max}]$:
\begin{align*}
    \widetilde{J}(s,\lambda) 
    := \begin{cases}
        \begin{aligned}[t]
         \widehat{J}(x,\phi,\lambda) \quad 
             \text{if } s\in\Scal,\; s=[x,1,0,\phi],
            \end{aligned}
            \\[1ex]
             \widehat{W}(x,\phi) + \lambda\psi
             \hfill \text{otherwise},
        \end{cases} \nonumber
    \end{align*}
where $\widehat{J}$ satisfies
\begin{align}\label{eq:approx_bellman_J}
    &\widehat{J}(x,\phi,\lambda) = \min_{u\in\Ucal} \Bigl\{\phi\ell(x,u) \nonumber
    \\& + \E_d\!\left[\widehat{J}(f(x,u,d), \gamma\phi,\lambda)\mid x,\phi, u,  f(x,u,d)\in\Xcal\right] \nonumber
    \\& \quad \cdot \P(f(x,u,d)\in\Xcal \mid x,u) \nonumber 
    \\& + \left(\E_d\!\left[\widehat{W}(f(x,u,d), \gamma\phi)\mid x,\phi, u, f(x,u,d)\not\in\Xcal\right] + \lambda \right) \nonumber
    \\& \quad \cdot \P(f(x,u,d)\not\in\Xcal\mid x,u) \Bigr\},
\end{align}
and $\widehat{W}$ satisfies
\begin{align}\label{eq:approx_bellman_W}
    &\widehat{W}(x,\phi) = \min_{u\in\Ucal} \Bigl\{\phi\ell(x,u) + \E_d[\widehat{W}(f(x,u,d),\gamma\phi)\mid x,\phi,u]\Bigr\}.
\end{align}
As in Step \ref{state:targetJ} of Algorithm 2, in \eqref{eq:approx_bellman_J} we have used the law of conditional expectation to determine whether the next state belongs to $\Xcal$.
Then, we can extract the optimal dual variable $\lambda^\star(s_0)\in[0, \lambda_\text{max}]$, associated to a certain initial state $s_0\in\Scal$, with $(\xi_0, \psi_0) = (1, 0)$ by solving 
\begin{align}\label{eq:lambda_opt}
    \lambda^\star(s_0) = \arg\max_{\lambda\in[0, \lambda_\text{max}]} (\widehat{J}(x_0, 1, \lambda) - \lambda\veps).
\end{align}
Note that $J(s_0, \lambda) - \lambda\veps$ is the dual function of \eqref{cc_ocp_dual}; therefore, it is concave. Hence, $\widehat{J}(s_0, \lambda) - \lambda\veps$ should also be approximately concave provided that the NN $\widehat{J}$ is a good approximation of $J$. Hence, \eqref{eq:lambda_opt} can be solved efficiently, e.g.,  via automatic differentiation. 
Specifically, by setting the learning rate in a way that \eqref{eq:alpha_ass} is satisfied, the resulting dual ascent scheme to solve \eqref{eq:lambda_opt} reproduces the steps \eqref{eq:primal_update}--\eqref{eq:dual_update}, with the difference that the NN $\widehat{J}$ is used in place of $J$.

Then, after the optimal multiplier $\lambda^\star(s_0)$ is found for the chosen closed-loop initial condition $s_0$, the optimal input at time step $t\in\Z_{\geq0}$ can be found as a minimizer of the one-step-ahead problem
\eqref{eq:approx_bellman_J} if $\xi_t=1, \psi_t=0$, and of \eqref{eq:approx_bellman_W} if $\xi_t=0$.

Solving this problem can be significantly cheaper than using MPC-based approaches, which typically consider a long prediction horizon in the online control phase. In our case, a one-step problem is sufficient, since $\widetilde{J}$ already approximates the value of the infinite-horizon problem. 

\section{Numerical experiments} \label{sec:experiment}
In this section, we validate\footnote{Code available at https://github.com/fracordi/chance-constrained-adp} our approach on a numerical example and compare it with an MPC scheme that approximates the joint chance constraint using a randomized approach \cite{campi2008exact, prandini2012randomized}.

\subsection{Setup}
We consider a two-state unicycle \cite{schmid2025computing}, with dynamics described by 
\begin{align}\label{eq:dubin}
    \begin{cases}
        p_{x, t+1} = p_{x, t} + 0.8 \cos(u_{t}) + d_{x, t}
        \\p_{y, t+1} = p_{y, t} + 0.8 \sin(u_{t}) + d_{y, t},
    \end{cases}
\end{align}
where the state vector $x$ is represented by the position of the unicycle, with $p_x\in [-10,  10]$ and $p_y \in[-5, 10]$. The control input $u\in[-\pi, \pi]$ is the angle, and the unicycle moves at a constant driving speed of 0.8. The disturbances follow a truncated Gaussian distribution, with 0 mean and covariance matrix diag$(0.2^2, 0.2^2)$, truncated to the interval $[-0.3, 0.3]\times[-0.3, 0.3]$. The cost matrices are chosen as $Q=\text{diag}(1, 1)$ and $R=0.1$, and the reference for the state is $x_\text{ref}=[0, 0]^\top$ and for the input it is $u_\text{ref}=0$. Then, the stage cost is $\ell(x,u) = \|x-x_\text{ref}\|_Q^2 + \|u-u_\text{ref}\|_R^2$, with a discount factor of $\gamma=0.95$. The stage cost is clipped at a large-enough constant for points out of the state space, in a way that Assumption \ref{ass:bounded_cost} is satisfied.

The unicycle has to satisfy an obstacle-avoidance specification, where the obstacle is a rectangle in the state space described by $(p_x, p_y)\in[-2, 10]\times[3, 6]$ (see Figure \ref{fig:traj}). This has to be achieved in a probabilistic sense, in which a trajectory is considered unsafe if, at any time step, it hits the obstacle. This should occur with probability at most $\veps=0.1$.

To implement our approach, we follow Algorithms \ref{alg:learn_W} and \ref{alg:learn_J}, using 200 state samples obtained by sampling from the state space, excluding the obstacle. Among these samples, 100 are uniformly chosen at random from the safe set to ensure good coverage of both dimensions. Then, 100 samples are taken closer to the obstacle to improve the accuracy of learning the safety specification. Specifically, we draw 100 samples uniformly at random from the area around the obstacle, with points having a distance, measured in the infinity norm, from the obstacle less than or equal to 1. Then we design an NN architecture with 3 hidden layers, each containing 64 hidden units, to learn $\widehat{W}$ and $\widehat{J}$ using Algorithms \ref{alg:learn_W} and \ref{alg:learn_J}. For the latter, we select $\lambda\in\{0, 50, 100, ..., 350, 400\}$, where $\lambda_\text{max}=400$ has been selected via trial and error, since, for all initial conditions that we have tested, the optimal dual variable has always achieved a lower value. We use AdamW optimizer \cite{loshchilov2017fixing}, and the Gaussian Error Linear Unit (GELU) activation function, which is a smooth version of the classical Rectified Linear Unit (ReLU). Then, the one-step-ahead problems in Algorithms \ref{alg:learn_W} and \ref{alg:learn_J} are solved by approximating the expectations using 400 uncertainty samples, which is enough to provide a good empirical approximation of the true expectation. Lastly, the minimum with respect to $u$ is found over 100 input samples, which provides a good accuracy to cover the one-dimensional interval $[-\pi, \pi]$.

Regarding the MPC scheme used for comparison, we need to introduce suitable approximations to the chance constraint. This is done in two steps: First, in the open loop prediction, the chance constraint is approximated by its finite-horizon version, with horizon $H_\text{pred}\in\Z_{>0}$; second, we replace the resulting probabilistic constraint by its randomized approximation, i.e., by enforcing $N_\text{rnd}\in\Z_{>0}$ hard constraints. This is in accordance with the theoretical results in \cite{campi2008exact, prandini2012randomized}, in which the risk of a constraint violation for a disturbance realization possibly out of the sample set decreases with $N_\text{rnd}$. Therefore, when solving the MPC problem at time step $t$, the chance constraint is approximated by $x_{t+k}^{(i)} \in \Xcal, \forall k\in\{1,...,H_\text{pred}\}, \forall i\in\{1,...,N_\text{rnd}\}$. Then, to enforce the obstacle avoidance specification, we equivalently define the safe set $\Xcal$ to be the union of three polytopes, specifically: $\Xcal = \Pcal_1 \cup \Pcal_2 \cup \Pcal_3$, with $\Pcal_1=[-10, 10]\times[-5, 3], \Pcal_2=[-10, -2]\times[3, 6]$, and $\Pcal_3=[-10, 10]\times[6, 10]$. In Figure \ref{fig:traj}, this corresponds to the set indicated by the dashed line excluding the obstacle.

The resulting MPC problem is a nonlinear optimization problem with a union of polytopic constraints, which can be cast as a mixed-integer nonlinear program, by associating a binary variable $y_{kip}=1$ if and only if $x_k^i$, i.e., the state at predicted time step $k$ according to uncertainty realization $i$, is in polytope $\Pcal_p$, $\forall k\in\{1,...,H_\text{pred}\}, \forall i\in\{1,...,N_\text{rnd}\}, p\in\{1,2,3\}$, and 0 otherwise, with the condition that $\sum_p y_{kip} = 1, \forall k\in\{1,...,H_\text{pred}\}, \forall i\in\{1,...,N_\text{rnd}\}$. In addition, we add slack variables to prevent infeasibility when the system state lies in the region defined by the obstacle, due to potential constraint violations. The slack variables are then penalized in the objective function. The nonlinear MPC scheme is implemented using CasADi \cite{andersson2019casadi}, version 3.7.2, with solver Bonmin, version 1.8.9.

In the following, we test our approach and the randomized MPC scheme by means of a Monte Carlo closed-loop simulation, where we evaluate, in particular, the closed-loop cost and constraint violation, defined as
\begin{align*}
    L_\text{cl} & := \frac{1}{N_\text{mc}}\sum_{i=1}^{N_\text{mc}} \sum_{t=0}^{T_\text{cl}} \gamma^t\ell(x^i_t, u^i_t)
    \\ R_\text{cl} & := \frac{1}{N_\text{mc}}\sum_{i=1}^{N_\text{mc}} \mathbf{1}_{\left(\exists t \in \{0,...,T_\text{cl}\}:\, x_t^i \not\in\Xcal \right)}
\end{align*}
where, in this case, $x^i_t$ represents the state of the closed-loop system according to the sample trajectory $i$, for $t\in\{0,..., T_\text{cl}\}, i\in\{0,..., N_\text{mc}\}$.

\subsection{Results and comparisons}

First, we analyze the results of the function approximator $\widehat{J}$ resulting from Algorithm \ref{alg:learn_J}. Figure \ref{fig:Jhat} depicts $\widehat{J}(x,\phi)-\lambda\veps$, i.e., the dual function in \eqref{cc_ocp_dual}, for $\phi=1$ and $x_0\in\{[4, 7]^\top, [5, 8]^\top, [6, 9]^\top\}$, i.e., three different initial conditions above the obstacle, and gradually farther from the origin. We observe that the resulting function is (approximately) concave, with points farther from the setpoint being more costly. The concave shape is expected, since $\widehat{J}-\lambda\veps$ approximates the dual function, defined in Section \ref{sec:duality}, which is the pointwise infimum of affine functions and therefore it is concave.

\begin{figure}
    \centering
    \includegraphics[width=0.8\linewidth]{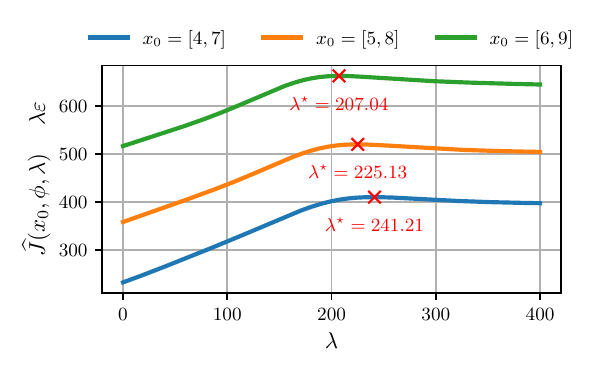}
    \caption{$\widehat{J}-\lambda\veps$ for different initial states and $\phi=1$. The red cross represents the optimal $\lambda$ for a given state.}
    \label{fig:Jhat}
\end{figure}

We now analyze the closed-loop performance resulting from our scheme, in terms of cost and constraint satisfaction. 
Figure \ref{fig:traj} shows the closed-loop trajectories of the Monte Carlo simulation, starting from $x_0=[5, 8]$, with $T_\text{cl}=30$ and $N_\text{mc}=200$. Our learning-based approach (left figure) yields consistent results, both in terms of
probabilistic feasibility and of closed-loop performance. The obstacle is hit with a small probability, computed empirically, of $R_\text{cl} = 0.045$, which is smaller than $\veps=0.1$, and $L_\text{cl} = 528.3$. Whenever the obstacle is hit, the trajectory continues along the path with lower cost, consistent with our interpretation in Section \ref{sec:value_structure}, which asserts that once the constraint is violated for a sample trajectory, the optimal closed-loop behavior is to follow the path that minimizes the cost.

\begin{figure}
    \centering
    \includegraphics[width=0.99\linewidth]{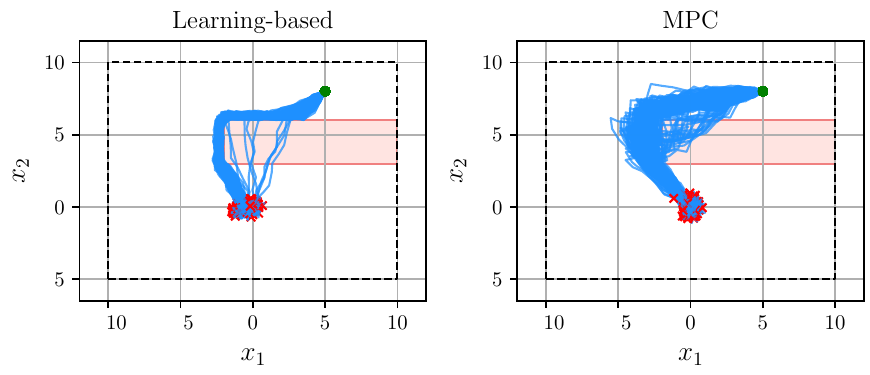}
    \caption{Closed-loop trajectories of our learning-based approach (left) versus the randomized MPC scheme (right, with $H_\text{pred}=6, N_\text{rnd}=25$), starting from $x_0=[5,8]^\top$, with $N_\text{mc}=200$.}
    \label{fig:traj}
\end{figure}

A similar behavior is observed for other initial conditions around the obstacle, for which the empirical closed-loop constraint violation is shown in Figure \ref{fig:x0_validation}.
Here, we have obtained a grid of the state space with 15 points along each axis by maintaining a margin of 0.5 from the boundaries. We have then discarded the points on the obstacle, resulting in 185 points. We can see that, for most of the initial conditions we have tested, our method gives an empirical probability of constraint violation close to $\veps$.  Specifically, 168 are feasible (i.e., with a violation in the range $[0, 0.1]$), 11 yield a violation in the range $(0.1, 0.2]$ and 6 in the range $(0.2, 0.71]$. Note that deviations may arise from approximation errors in the learning scheme or from a limited dataset, which can be addressed by sampling more points in critical regions around the obstacle.

\begin{figure}
    \includegraphics[width=0.95\linewidth]{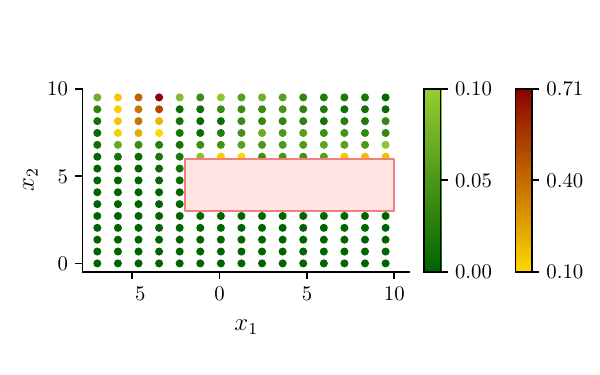}
    \vspace{-0.5cm}
    \caption{Color map representing the empirical constraint violation for initial states around the obstacle when using our approach, computed with $N_\text{mc}=500$ trajectories for each initial condition.}
    \label{fig:x0_validation}
\end{figure}

We now consider again $x_0=[5, 8]$, and run a closed-loop comparison with the randomized MPC approach, with $N_\text{rnd}=25$. For computational tractability, we set a maximum solver time of 10 seconds.
To address the obstacle-avoidance specification, a very long prediction horizon would be needed to find the path with the minimum cost while avoiding the obstacle. To alleviate the computational burden, we split the control task into two phases: we first require the system to reach the intermediate reference $[-3,6.5]^\top$ on the left side of the obstacle; then, once the state of the system is sufficiently close to this intermediate reference, we set $[0,0]^\top$ as the final target. This is based on the intuition that the optimal policy resulting from the MPC scheme, which enforces $N_\text{rnd}$ hard constraints, will certainly avoid the obstacle on the left, since this path is the only feasible one. By splitting the control task into two phases, we see that a prediction horizon $H_\text{pred}=6$ is sufficient to solve the obstacle-avoidance problem.
The resulting Monte Carlo simulation yields $L_\text{cl}=672.0$ and $R_\text{cl}=0.17$ for $N_\text{rnd}=25$, and the trajectories are shown in Figure \ref{fig:traj} (right figure). We observe that $N_\text{rnd}=25$ is insufficient to meet the safety specification with the desired probability; therefore, more samples are required in the randomized approach, thereby increasing computational effort. In addition, we observe more variability in the closed-loop trajectories of the randomized approach. This can be due to the solver time being limited to 10 seconds, which is most often reached when the system is near the obstacle. In such cases, a possibly suboptimal solution is obtained, which might degrade the closed-loop performance.
Moreover, the performance of the randomized approach is more conservative than ours, since it yields $L_\text{cl}=672.0$, which is greater than $L_\text{cl}=528.3$ in our approach. This is because the policy resulting from our approach follows the shortest path when a violation occurs, confirming the interpretation given in Section \ref{sec:value_structure}, whereas the policy resulting from the randomized MPC controller does not. 
For both our approach and the MPC one,  we observe that the system position converges to a neighborhood of $[0,0]^\top$; therefore, the reference $x_\text{ref}$ is correctly tracked. However, in our approach, there is a small tracking error. This is mostly due to the discount factor, since the cost of future states is weighted less in the value function. Hence, the tracking error can be reduced by increasing $\gamma$.

Last, we compare the two approaches in terms of solver time. As mentioned before, the maximum solver time for the randomized MPC approach is set to 10 seconds. For example, by choosing $x_0=[5, 8]^\top$, this limit is achieved for some time steps when the system is in the proximity of the obstacle, whereas the average solver time is 2.94 seconds. When using our approach, the time required to solve the one-step-ahead problem via sampling over 100 input samples is 0.01 seconds, which is significantly lower than that of the MPC approach. 

\section{Conclusions} \label{sec:conclusions}
In this paper, we have proposed a computationally efficient way to solve optimal control problems with infinite-horizon chance constraints. By means of the Lagrange dual framework and of an appropriate state augmentation, we have formulated an unconstrained Markov  control problem over the augmented state space that is equivalent to the original one. This problem enjoys several theoretical properties, which allow to solve it by means of classical unconstrained reinforcement learning algorithms. A dedicated learning scheme to approximate the dual function allows us to consider continuous state-input spaces, and the resulting approach outperforms state-of-the-art methods in terms of performance and computational complexity. 

The most relevant topics for future work consist in addressing the weaknesses of our learning scheme, e.g., reducing or bounding the approximation error, investigating convergence properties of the approximate value iteration scheme, and exploring methods to improve the scalability of our learning-based approach with respect to the system dimension.

\section*{References}
\bibliographystyle{IEEEtran}

\bibliography{bibliography.bib}

\begin{IEEEbiography}[{\includegraphics[width=1in,height=1.25in,clip,keepaspectratio]{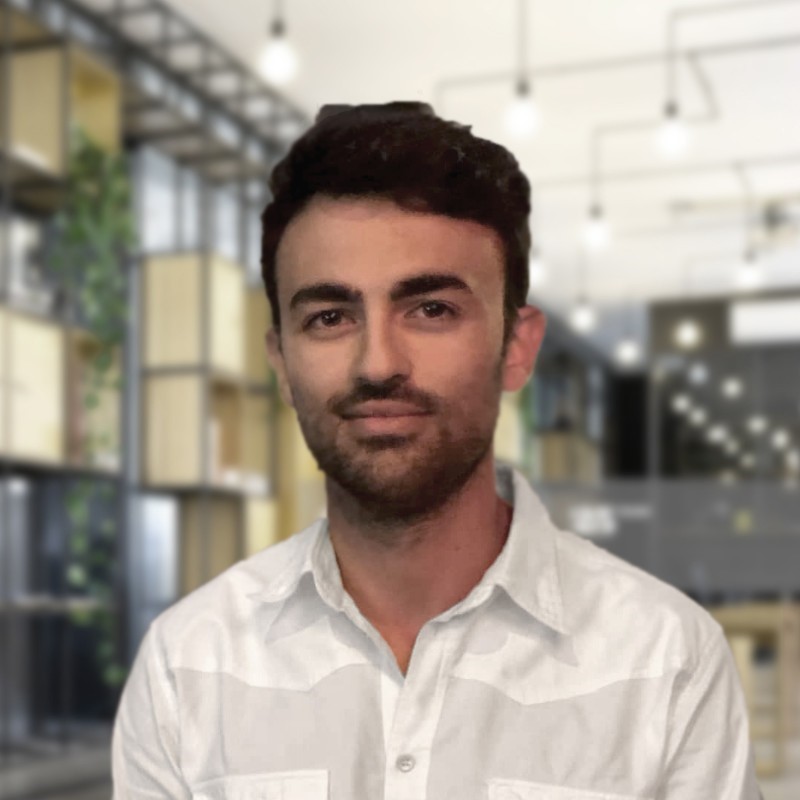}}]{Francesco Cordiano} received the B.Sc. degree in automation engineering
from Politecnico di Milano, Italy, and the M.Sc.
degree in robotics, systems, and control from ETH Zurich, Switzerland, in 2019 and 2022, respectively. 
He is currently a PhD candidate at the Delft Center
for Systems and Control, Delft University of
Technology, The Netherlands.

His current research interests include
stochastic optimization, reinforcement learning,
and model predictive control of hybrid systems.
\end{IEEEbiography}

\begin{IEEEbiography}[{\includegraphics[width=1in,height=1.25in,clip,keepaspectratio]{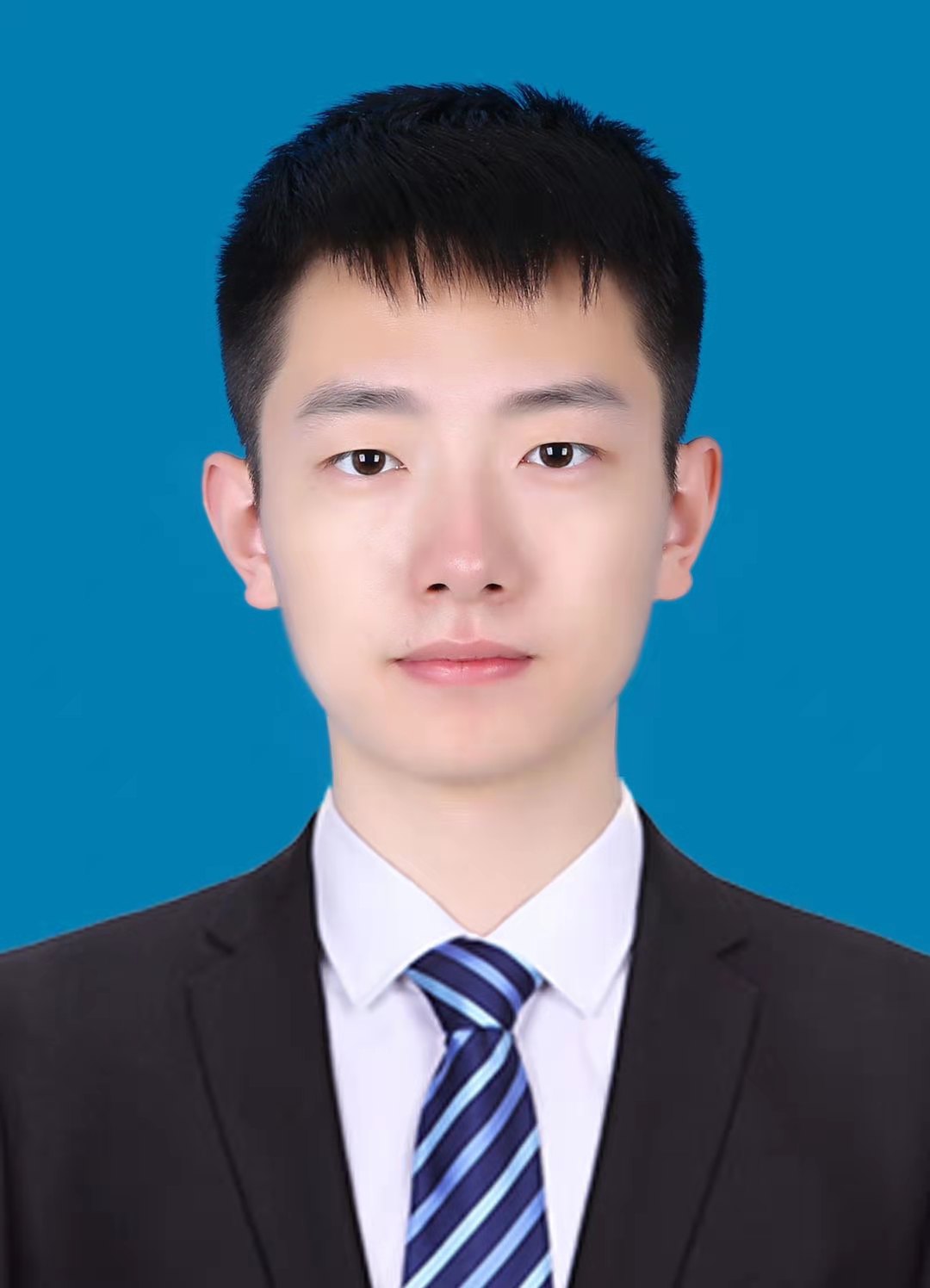}}]{Kanghui He} is a postdoctoral researcher in the Department of Engineering Science, University of Oxford, U.K. He received his PhD from the Delft Center for Systems and Control at Delft University of Technology, the Netherlands, in 2026. He received the M.Sc. degree from the Department of Flight Dynamics and Control at Beihang University in 2021 and the B.Sc. degree from the School of Mechanical Engineering and Automation at Beihang University in 2018. He was a research assistant in the Department of Automation, Tsinghua University. His research interests include learning-based control, model predictive control, optimization, and their applications in mobile robots.
\end{IEEEbiography}

\begin{IEEEbiography}[{\includegraphics[width=1in,height=1.25in,clip,keepaspectratio]{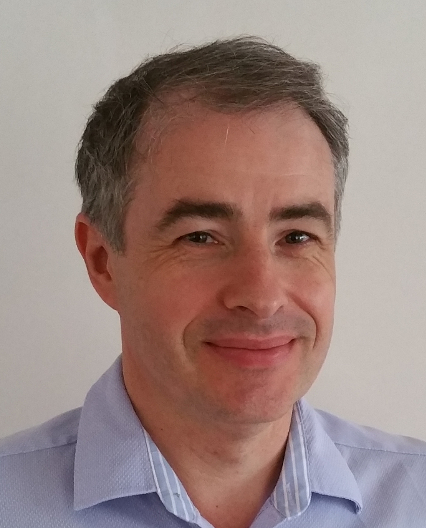}}]{Bart De Schutter}
(IEEE member since 2008,
senior member since 2010, fellow since 2019) is
a full professor and head of department at the
Delft Center for Systems and Control of Delft
University of Technology in Delft, The Netherlands.

Bart De Schutter is senior editor of the IEEE
Transactions on Intelligent Transportation Systems. His current research interests include integrated learning- and optimization-based control and decision making,
multi-level and multi-agent control, and control of
hybrid systems.
\end{IEEEbiography}

\end{document}